\documentclass[letter,11pt]{article}
\usepackage[margin=1in]{geometry}
\usepackage{amsmath,multicol,graphicx}
\usepackage{amsfonts}
\usepackage{color}
\usepackage{bbm}
\usepackage{cite}
\usepackage{amssymb}
\usepackage{algorithm}
\usepackage{algorithmic}
\usepackage{tabularx}
\usepackage{hyperref}
\usepackage[utf8]{inputenc}
\usepackage[english]{babel}

\usepackage{amsthm}
\usepackage{subcaption}
\usepackage[font={footnotesize}]{caption}
\usepackage{url}
\usepackage{mathtools}
\usepackage{comment}
\usepackage{array,booktabs,dcolumn,calc}
\usepackage{float}
\DeclareRobustCommand*\cal{\@fontswitch\relax\mathcal}

\usepackage{nicefrac}
\usepackage{nicematrix}

\theoremstyle{definition}
\newtheorem{definition}{Definition}[]
\newtheorem{assumption}{Assumption}[]
\newtheorem{proposition}{Proposition}[]
\newtheorem{theorem}{Theorem}[]

\newtheorem{lemma}[]{Lemma}

\DeclareMathOperator{\E}{\mathbb{E}}
\DeclareMathOperator*{\argmax}{arg\,max}

\def\E{{\mathbb E}}

\newcommand{\argmin}{\arg\!\min}

\def\S{{\cal S}}
\def\T{{\cal T}}

\def\Xs{{\cal X}}
\def\R{{\cal R}}
\def\N{{\cal N}}

\let\oldref\ref
\renewcommand{\ref}[1]{(\oldref{#1})}
\newcommand{\RNum}[1]{\uppercase\expandafter{\romannumeral #1\relax}}
\newcolumntype{C}{>{\centering\arraybackslash}b{\widthof{positions}}}
\newcolumntype{d}{D{.}{.}{-2}}

\makeatletter
\renewcommand{\fnum@figure}{Fig.~\thefigure}
\makeatother

\usepackage{amsmath}
\usepackage{amssymb}
\usepackage{graphicx}
\usepackage{subcaption}
\usepackage{nicefrac}
\usepackage{nicematrix}
\usepackage{color}
\usepackage{amsfonts}
\usepackage{mathtools}
\usepackage{tabularx}
\usepackage{algorithm}

\usepackage{algorithmic}

\def\E{{\mathbb E}}

\def\S{{\cal S}}
\def\T{{\cal T}}

\def\Xs{{\cal X}}
\def\R{{\cal R}}
\def\N{{\cal N}}

\makeatletter
\DeclareRobustCommand*\cal{\@fontswitch\relax\mathcal}
\makeatother
\title{\vspace{-8mm}\textbf{Randomized Greedy Methods for Weak Submodular Sensor Selection with Robustness Considerations}}

\date{}
\vspace{-0.05in}
\author{Ege C. Kaya, Michael Hibbard, Takashi Tanaka, \\ Ufuk Topcu, Abolfazl Hashemi\thanks{Ege C. Kaya and Abolfazl Hashemi are with the Elmore Family School of Electrical and Computer Engineering, Purdue University, West Lafayette, IN, USA. Michael Hibbard, Takashi Tanaka and Ufuk Topcu are with the Department of Aerospace Engineering and Engineering Mechanics, University of Texas at Austin, Austin, TX, USA. This work was presented in part at the 2023 American Control Conference (ACC) \cite{hibbard_hashemi_tanaka_topcu_2023}. The full work was published in Automatica, 2025 \cite{kaya2025randomized}. This work is supported in part by DARPA grant HR00112220025.}}
\begin{document}
\maketitle
\begin{abstract}
We study a pair of budget- and performance-constrained weak-submodular maximization problems. For computational efficiency, we explore the use of stochastic greedy algorithms which limit the search space via random sampling instead of the standard greedy procedure which explores the entire feasible search space. We propose a pair of stochastic greedy algorithms, namely, \textsc{Modified Randomized Greedy (MRG)} and \textsc{Dual Randomized Greedy (DRG)} to approximately solve the budget- and performance-constrained problems, respectively. For both algorithms, we derive approximation guarantees that hold with high probability. We then examine the use of DRG in robust optimization problems wherein the objective is to maximize the worst-case of a number of weak submodular objectives and propose the \textsc{Randomized Weak Submodular Saturation Algorithm (Random-WSSA)}. We further derive a high-probability guarantee for when \textsc{Random-WSSA} successfully constructs a robust solution. Finally, we showcase the effectiveness of these algorithms in a variety of relevant uses within the context of Earth-observing low Earth orbit satellite constellations which estimate atmospheric weather conditions and provide Earth coverage.
\end{abstract}
\section{Introduction}\label{sec:intro}
In sensor selection problems, especially those involving an extensive network of sensors, one often needs to select an optimal subset for a given task rather than using all available due to restrictions on resources.
Such restrictions, usually modeled in terms of constraints on the cardinality or the cost of feasible solutions, introduce the need to design principled decision-making algorithms for the creation of task-optimal selections of sensors that deliver high performance while adhering to the inherent resource, performance and robustness constraints.
These problems occur naturally in many areas such as power system monitoring~\cite{damavandi2015robust,liu2016towards}, economics and algorithmic game theory \cite{cardinality,revenue,gametheory}, and finance~\cite{attigeri2019feature}. In this work, we motivate the problem setting by considering low Earth orbit (LEO) constellations of Earth-observing satellites~\cite{richards2001distributed}.


The problem of choosing the optimal set of sensors is a computationally challenging combinatorial optimization that is known to be NP-hard~\cite{williamson2011design}.
Because of this complexity and theoretical barrier, a considerable amount of research has delved into creating heuristic algorithms aimed at providing approximate solutions for the problem.

Particularly well-known is the scenario where the objective can be formulated using a function that is both monotone nondecreasing and submodular, i.e., a set function that has the familiar property of demonstrating \textit{diminishing marginal returns}. When sensor selection is made in the presence of a cardinality constraint on the number of sensors, a proven result is that the \textsc{Greedy} algorithm achieves an approximation factor of $(1-1\slash e)$ in comparison to the optimal solution \cite{nemhauser1978analysis}.
Specifically, at each iteration, the \textsc{Greedy} algorithm considers each element in the set of remaining sensors, evaluates them in terms of marginal gain, and adds the element with the highest marginal gain to the set of selected elements.

Modified versions of \textsc{Greedy} have been developed for addressing sensor selection problems with the more general case of budget constraints~\cite{wolsey1982analysis}. In this case, each sensor incurs a cost, and the decision-maker is restricted by a predetermined budget for their selections. The dual of this problem, where the objective is to identify the most cost-effective set of sensors to meet a specific performance constraint, has also been considered~\cite{khuller1999budgeted}.


In situations involving a large number of sensors, the computational demands of \textsc{Greedy} become impractical due to the necessity of evaluating the marginal return of every remaining sensor at each iteration.
In~\cite{mirzasoleiman2015lazier}, a computationally cheaper alternative is proposed, often referred to as the \textsc{Stochastic} or \textsc{Randomized Greedy} algorithm in the literature. This approach limits the function evaluations to a randomly sampled subset of the remaining sensors during each iteration and enjoys theoretical assurances concerning the expected performance of the constructed solution relative to the optimal solution.
Expanding upon this randomized algorithm,~\cite{hashemi2020randomized} accounts for situations where the performance objective is defined by a weak submodular function, which is a function that is either truly submodular or whose maximum violation of submodularity is bounded.
Moreover, the authors derive a theoretical guarantee that holds with high probability for the performance of the randomized greedy algorithm in the weak submodular regime. However, it is important to note that the setting they consider is restricted solely to the cardinality-constrained case and does not take into account the nontrivial generalization of this setting to one involving the presence of a budget constraint and nonuniform selection costs.


In this work, we draw inspiration from the previously aforementioned work to propose a pair of novel extensions to the \textsc{Randomized Greedy} algorithm for use in budget- and performance-constrained problems, namely, the \textsc{Modified Randomized Greedy} (MRG) and \textsc{Dual Randomized Greedy} (DRG) algorithms, respectively.
For both cases, we assume that the performance objective is characterized by a weak submodular function and that the sensor costs are additive and possibly nonuniform.
To the best of our knowledge, the use of randomized greedy algorithms in these two scenarios is a novel approach. 
In the ensuing discussion, we further derive theoretical guarantees that hold with high probability on the approximation factors for both algorithms.
As expected, in the case where the performance objectives are not weak submodular but truly submodular, and where the standard greedy procedure is used instead of a randomized greedy approach, both of the derived theoretical guarantees reduce to the standard results proven in the literature.

We further examine the use of the proposed DRG algorithm in the context of robust multi-objective submodular maximization where the robustness is captured by maximizing the value of the worst objective. This notion of worst-case robustness extends the setting explored in \cite{krause2014submodular} which proposes the \textsc{Submodular Saturation Algorithm} (SSA) in the presence of multiple monotone nondecreasing submodular objective functions. We propose a modified version of the algorithm, which we name \textsc{Randomized Weak Submodular Saturation Algorithm (Random-WSSA)}, and likewise derive a theoretical guarantee on the performance of the algorithm that holds with high probability.



The remainder of the paper is organized as follows:
In Section~\ref{sec:motiv}, we motivate our work in the context of LEO satellite constellations while in Section~\ref{sec:back} we provide the required preliminary information used in the rest of the work.
In Section~\ref{sec:mart}, we introduce the MRG and the DRG algorithms and derive lower bounds on their approximation ratios.
In Section~\ref{sec:robust}, we introduce \textsc{Random-WSSA} for the multi-objective robustness problem and provide theoretical guarantees on its success.
Finally, in Section~\ref{sec:examples}, we corroborate our theoretical results by providing numerical examples demonstrating the efficacy of the three proposed algorithms for sensor selection in Earth-observation tasks, as well as in scenarios involving multi-objective robustness.
\section{Motivation: Satellite Sensing Networks}\label{sec:motiv}
We start by presenting an illustrative motivating example within the field of LEO satellite constellations.
%
Over the past few years, there has been a growing interest in satellite mission designs to use a cluster of simpler satellites rather than relying on a single intricate satellite. This trend has made the use of CubeSat-based designs~\cite{poghosyan2017cubesat} commonplace, such as the constellation depicted in Figure~\ref{fig:motivating_example}.
Constellations built on such concepts provide mission designers with various advantages. These advantages include cost savings in the design and launch expenses~\cite{wagner2021distributed}, increased redundancy to address potential failures, and improved temporal resolution in observations due to more frequent revisits~\cite{cahoy2017initial}.
Some examples of missions that have achieved success with this design philosophy are the NASA TROPICS mission~\cite{blackwell2018overview} and the Planet Labs Flock Constellation~\cite{boshuizen2014results}.
The NASA TROPICS mission had as its primary objective to acquire temporally dense observations of weather phenomena, while the objective of the Flock Constellation was high-resolution imaging of the Earth's surface.
Proposed uses for CubeSat-based missions include many types of Earth science-related applications~\cite{selva2012survey,poghosyan2017cubesat}, as well as disaster management, natural hazard response~\cite{mandl2016hyperspectral,santilli2018cubesat} and air traffic monitoring~\cite{nag2016cubesat,wu2021multiple}.

Naturally, as the number of LEO constellations, as well as the number of satellites in the constellation increases, more human operators are needed for their supervision and operation.
One frequent mode of operation may involve the selection of only a subset of the satellites to use for sensing at each time step, suited to the needs of some specified task. 
Adhering to the assumption that the number of human operators required for the operation of such constellations is linearly correlated with the number of satellites in it~\cite{siewert1995system}, this human-operated approach may rapidly become cumbersome. Especially under emergency circumstances where quick decision-making is of great importance, having the selection decisions made by reliable automated methods could be much more effective than trusting it with human agents. Imagine the scenario of a sudden forest fire, which creates the need for the quick selection of an optimal set of satellites to image the affected site. In such a scenario, automation may mean the difference between a quick intervention and a belated one.
Hence, one motivating factor behind this work is to automate the process of selection by developing submodular optimization algorithms that efficiently select suitable subsets autonomously, with provable guarantees demonstrating that the produced solution will be trustworthy, robust, and replicable. 

\begin{figure}
    \centering
    \scalebox{0.9}{
    \includegraphics{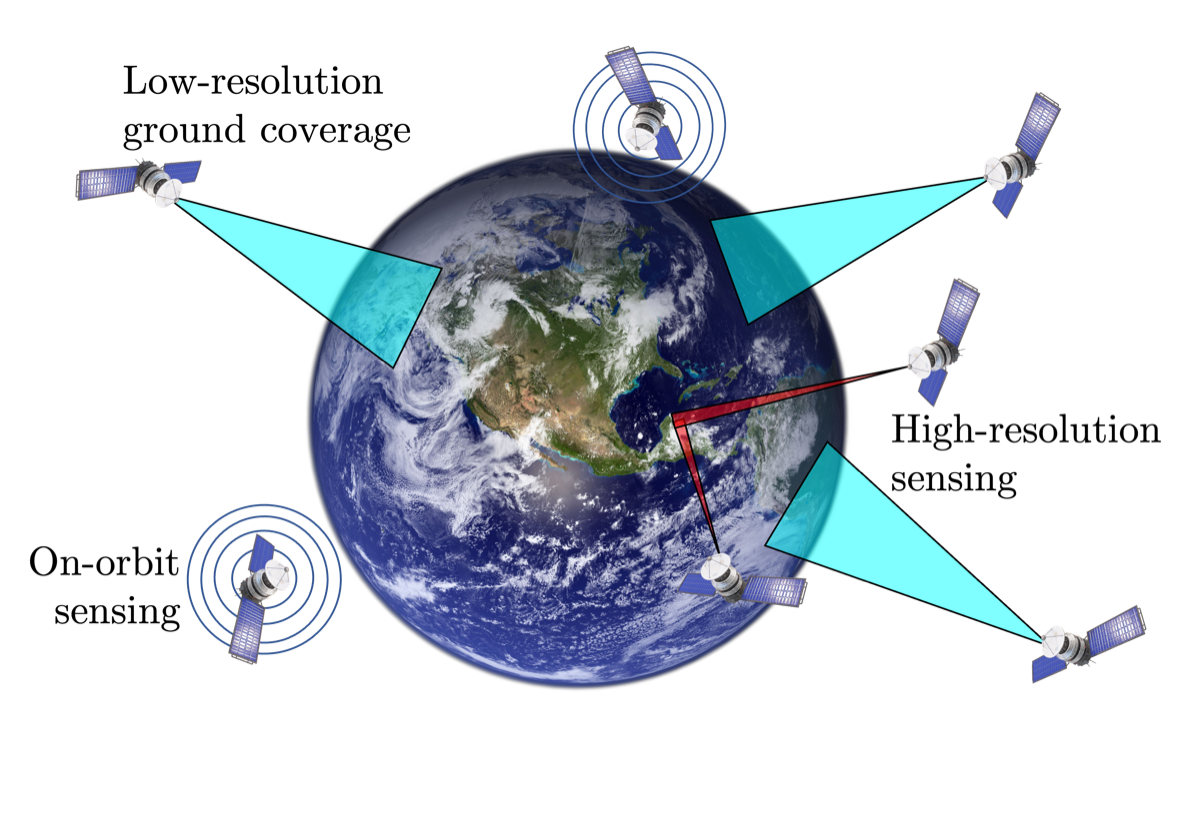}}
    \caption{It is often the case that only some of the satellites in a constellation are actively observing Earth at any given time. Furthermore, these satellites may be leveraged for multi-task missions.}
    \label{fig:motivating_example}
\end{figure}
\section{Preliminaries and Background}\label{sec:back}
In this section, we overview a number of relevant mathematical concepts for the ensuing algorithmic and theoretical developments.

\subsection{Notation} 

We denote the set of nonnegative real numbers by $\mathbb{R}_{\geq 0}$, finite sets using calligraphic notation $\mathcal{S}$, and the $n$-dimensional identity matrix by $I_{n}$.
For a finite set $\mathcal{S}$, we denote its cardinality by $|\mathcal{S}|$ and its power set by $2^{\mathcal{S}}$.
Finally, we compactly express $[n] := \{1,\ldots,n\}$.

\subsection{Submodular Maximization}

In this section, we provide an overview of submodular maximization. We begin with the following definitions.

\begin{definition}[Monotone nondecreasing set functions]
A set function $f\vcentcolon 2^\N \rightarrow \mathbb{R}$ is \textit{monotone nondecreasing} if for every $\S\subseteq \mathcal{T} \subseteq \N,$ we have $f(\S) \leq f(\mathcal{T}).$
\end{definition}
\begin{definition}[Submodularity]\label{def:submod}
A set function $f:2^\N\rightarrow \mathbb{R}$ is submodular if 
\begin{equation}
f(\S\cup \{j\})-f(\S) \geq f(\T\cup \{j\})-f(\T)
\end{equation}
for all subsets $\S\subseteq \T\subset \N$ and $j\in \N\setminus \T$. We also write $f_j(\S) \vcentcolon= f(\S\cup \{j\})-f(\S)$ for the marginal return of adding  $j$ to  $\S$.
\end{definition}

%
In the practical sense, the function $f$ is a performance objective (e.g., the additive inverse of the mean square error (MSE) \cite{hashemi2017randomized,hashemi2020randomized}) and $\N$ is the set of all sensors, out of which suitable selection should be made, commonly referred to as the \textit{ground set.}
%
In this paper, we consider the setting where each sensor $j \in \N$ has a cost of selection $c_{j} \in \mathbb{R}_{\geq 0}$, and the total cost of a selection $\S$ is denoted by $c(\S) \vcentcolon= \sum_{j \in \S} c_{j}$. Note that when $c_j=1$ for all $j\in \N$, we have $c(\S) = |\S|.$
Let $\mathcal{C} \vcentcolon= \{ c_{i} \, : \, i \in [|\N|] \}$ and define the ordered set $\Bar{\mathcal{C}} \vcentcolon= \{ \Bar{c}_{i} \in \mathcal{C} \, : \, \Bar{c}_{i} \leq \Bar{c}_{j} \,\, \forall \, j \geq i \}$ such that $\Bar{\mathcal{C}} = \mathcal{C}$.

For a monotone nondecreasing submodular function $f$ which is, in addition, normalized (i.e., $f(\emptyset)=0$), we are interested in solving the following budget-constrained combinatorial optimization problem:
\begin{equation}\label{eq:fmax}
\begin{gathered}
\max_{\S\subseteq\N} f(\S)\\
\text{s.t.}\;  c(\S)\leq B,
\end{gathered}
\end{equation}
where $B$ denotes the budget bound. 
By a reduction to the well-known set cover problem, \eqref{eq:fmax} is known to be NP-hard \cite{feige1998threshold,williamson2011design}. Note that when $c_j=1$ for all $j\in \N$, \eqref{eq:fmax} reduces to the cardinality-constrained observation selection problem studied in \cite{hashemi2017randomized,hashemi2020randomized}.
In this special case, the simple \textsc{Greedy} algorithm that iteratively selects the remaining element with the highest marginal gain
\begin{equation}
j^\ast \in \argmax_{j \in \N\setminus \S^{(i-1)}} f_j(\S^{(i-1)}),
\end{equation}
where $\S^{(i-1)}$ is the subset selected after the $(i-1)$th iteration, satisfies the optimal worst-case approximation ratio $f(\S) \geq \left(1-1/e\right) f(\S^\star)$, in which $\S$ is the subset selected by \textsc{Greedy}, and $\mathcal{S}^{*}$ is an optimal solution of~\eqref{eq:fmax}~\cite{nemhauser1978analysis}.
%
%
The simple extension of this result to the general, budget-constrained case is done through the \textsc{Modified Greedy} (MG) algorithm, which
instead selects the remaining element with the highest the marginal-gain-to-cost ratio $f_{j}(\S^{(i-1)})\slash c_{j}$
and adds it to the constructed solution as long as this addition does not violate the budget constraint.
Reminiscent of the well-known result for the \textsc{Greedy} algorithm, MG is shown to achieve a $\frac{1}{2}(1-1\slash e)$ worst-case approximation ratio~\cite{khuller1999budgeted}.
%

\subsection{Weak Submodularity}
In many sensor selection tasks, it is observed that $f$ is not strictly submodular but shows similar behavior to submodularity under certain conditions.
%
%
We will call such functions {\it weak submodular}, and capture their violation of submodularity through the quantity we will call the \textit{weak-submodularity constant}~\cite{zhang2016submodular,chamon2017approximate,hashemi2019submodular}.
\begin{definition}[Weak-submodularity constant (WSC)]\label{def:weak-sub}
The  WSC of a monotone nondecreasing set function $f$ is defined as
 \begin{equation}\label{eq:weak-sub}
 w_{f} \vcentcolon= {\max_{(\S,\T,j)\in \tilde{\N}}\frac{f_j(\T)}{ f_j(\S)}},
 \end{equation}
where $\tilde{\N} = \{(\S,\T,j):\S \subseteq \T \subset \N, j\in \N \setminus \T\}$. 
\end{definition}
Intuitively, the WSC of $f$ can be thought of as a measure of quantifying the maximum violation of the marginal diminishing returns property of submodularity. 
%
Note that with this definition, we will call any set function $f$ weak submodular if it is monotone nondecreasing and its WSC is bounded. Further note that by Definition \ref{def:submod}, a set function $f$ is submodular if and only if its WSC satisfies 
$w_{f} \le 1$ \cite{das2011submodular,elenberg2016restricted,horel2016maximization}. We will generally assume throughout this work that $w_f\geq 1$ to emphasize that the objective function is typically weak submodular.

For a monotone nondecreasing function with a bounded WSC, we have the following proposition~\cite{hashemi2019submodular}:
\begin{proposition}\label{lem:curv}
Let $w_f\geq 1$ be the WSC of $f$, a normalized, monotone nondecreasing set function. Then, for two subsets $\S$ and $\T$ such that $\S\subset \T \subseteq \N$, it holds that
\begin{equation}\label{eq:prop_1}
f(\T)-f(\S)\leq w_f\sum\nolimits_{j\in \T\backslash \S}f_j(\S).
\end{equation}
\end{proposition}
%
%
%
It is worth noting that the notion of weak submodularity is not yet standardized in literature. Alternative notions of weak submodularity exist, such as those presented in \cite{das2011submodular,elenberg2016restricted,zhang2016submodular,horel2016maximization}. Such notions may simplify the derivation of the approximation bounds depending on the application at hand (see e.g., \cite{bian2017guarantees,chamon2017approximate,khanna2017scalable,hashemi2019submodular,ghasemi2019submodularity}). Using the same notion of weak submodularity as the one in this work,~\cite{das2011submodular} extends the theoretical results of \cite{nemhauser1978analysis} on \textsc{Greedy} to the case of weak submodular functions, obtaining a worst-case approximation ratio of $f(\S) \geq \left(1-e^{-1\slash w_f}\right) f(\S^\star)$.
\section{High-Probability Approximation Factors for Greedy Algorithms with Random Sampling}\label{sec:mart}
\begin{algorithm}[t]
	\caption{\textsc{Modified Randomized Greedy (MRG)}}
	\label{alg:mg}
\begin{algorithmic}[1]
		\STATE \textbf{Input:} Weak-submodular function $f$, ground set $\N$, budget bound $B$, cost function $c$, sample set sizes $r_{i}$.
		\STATE \textbf{Output:} Subset $\S_{\operatorname{mrg}}\subseteq \N$ with $c(\S_{\operatorname{mrg}})\leq B$.
		\STATE Initialize $\S^{(0)} \gets \emptyset$, $\Xs \gets \N$, $i \gets 0$
		\WHILE{$\Xs \neq \emptyset$}
		\STATE Form $\R^{(i)}$ by sampling $\min \{r_i,\lvert\Xs\rvert\}$ elements from $\Xs$ uniformly at random.
		\STATE $j_s^i \in \argmax_{j \in \R^{(i)}} \frac{f_j(\S^{(i)})}{c_j}$
		\IF{$c(\S^{(i)})+c_{j_s^i}\leq B$}
		\STATE $\S^{(i+1)} \gets \S^{(i)}  \cup \{j_s^i\}$
		\STATE $i\leftarrow i+1$
		\ENDIF
		\STATE $\Xs \leftarrow \Xs - \{j_s^i\}$
		\ENDWHILE
		\STATE $j_{\max} \in \argmax_{j \in \N} f(\{j\})$
		\STATE $\S_{\operatorname{mrg}} \gets \argmax_{\mathcal{S} \in \{ \mathcal{S}^{(i)},\{ j_{\max} \} \}} f(\S)$
        \STATE \textbf{return} $\S_{\operatorname{mrg}}$
	\end{algorithmic}
\end{algorithm}

Greedy algorithms rely on a simple philosophy and are relatively straightforward to implement. In practice, however, it is usually undesirable due to computational cost considerations to use the full \textsc{Greedy} algorithm which searches over the entire set of elements.
To overcome this computational burden, we propose two \textit{randomized} greedy algorithms that, at each iteration, consider only a randomly sampled subset of the remaining elements in the ground set.
As briefly explained in Section \ref{sec:intro}, these two algorithms will be used to tackle a budget-constrained weak-submodular sensor selection task, and its dual formulation, i.e., a performance-constrained weak-submodular sensor selection task. Furthermore, we derive performance guarantees that hold with high probability for both algorithms.
%


\subsection{Budget-Constrained Sensor Selection}
We start our high-probability analysis by studying how MRG (Algorithm~\ref{alg:mg}) performs in budget-constrained weak-submodular sensor selection problems. 
To this end, we first need to address a fundamental question: \textit{how should we choose the size of the sampling sets $\R^{(i)}$ in each iteration?} Note that we cannot directly utilize the strategies employed in the cardinality-constrained setting, i.e., setting $r_i = \frac{\lvert\N\rvert}{K}\log \frac{1}{\epsilon}$ according to \cite{mirzasoleiman2015lazier,hashemi2017randomized,hashemi2020randomized}, or adopt a progressively increasing schedule $r_i = \frac{\lvert\N\rvert}{K-i}\log \frac{1}{\epsilon}$ as outlined in \cite{hashemi2021performance,hashemi2021PSG}, where $K$ is the number of observations to select and $0<\epsilon<1$ is a parameter set by the user. To see the barrier in choosing the sampling set size $r_i$, note that due to the budget constraint, it is not known a priori how many observations MRG will end up selecting (so, $K$ is unknown).
At this point, although intuitively $r_i$ should depend on $B$ and each sensor cost $c_j$, it is not clear which strategy results in a nontrivial approximation bound. Hence, we adopt a constructive approach where we let $r_i = \frac{\lvert\N\rvert}{U}\log \frac{1}{\epsilon}$ for some $U>0$ and $0<\epsilon<1$ where $U$ is determined as part of our analysis.

We then proceed to follow and extend~\cite{hassidim2017robust,hashemi2020randomized}, and begin with a simplifying assumption that views each iteration of MRG as an approximation to \textsc{Modified Greedy} for the budget-constraint formulation \eqref{eq:fmax}, which can be obtained from Algorithm \ref{alg:mg} by setting the sample cardinality $r_i$ equal to $|\N|$. Essentially, \textsc{Greedy} is a special case of MRG in which the entire set of remaining sensors is sampled at each iteration. MRG is no longer able to consider every remaining element in the ground set at each iteration, so there is no guarantee of it successfully adding the remaining sensor whose marginal-gain-to-cost ratio is maximal as \textsc{Greedy} does. However, one can still view each iteration of MRG as adding an element satisfying
\begin{equation}\label{eq:thm-mar:mrg:mar}
\begin{aligned} 
   \frac{f_{j_{\operatorname{mrg}}^i}(\S^{(i)})}{c_{j_{\operatorname{mrg}}^i}} & \geq \eta^{(i+1)} \max_{j \in \N \backslash \S^{(i)}} \frac{f_j(\S^{(i)})}{c_j}  = \eta^{(i+1)} \frac{f_{j_{\operatorname{g}}^i}(\S^{(i)})}{c_{j_{\operatorname{g}}^i}},
\end{aligned}
\end{equation}
where the subscripts $\operatorname{mrg}$ and $\operatorname{g}$ refer to the sensor selected in the MRG and \textsc{Greedy} iterations, respectively, and $\eta^{(i)}\sim\mathcal{D}$ is a random variable drawn from some distribution $\mathcal{D}$ such that $\eta^{(i)}\in (0,1]$ for all $i \in \mathbb{N}$. 
Essentially, the key idea is that each iteration $i$ of MRG selects a sensor whose marginal-gain-to-cost ratio is at least $\eta^{(i+1)}$ times as good as the one that would be selected by \textsc{Greedy}.
In contrast to \cite{hassidim2017robust,hashemi2020randomized},  which consider i.i.d. draws of $\eta$s for the cardinality-constrained scenario, 
we exploit the fact that each iteration of MRG \textit{is} dependent on the selections made in the previous iterations. We then adopt the relaxed assumption that the stochastic process $\{\eta^{(i)}\}_{i=1}^\infty$ is a martingale~\cite{chung2006concentration}.
\begin{definition}[Martingale]\label{def:martingale}
    A stochastic process $\{Y_{i}\}_{i=1}^\infty$ is a martingale with respect to the stochastic process $\{X_{i}\}_{i=1}^\infty$ if, for all $n \in \mathbb{N}$, $\mathbb{E}[\lvert Y_{n}\rvert] < \infty$ and $\mathbb{E}[Y_{n+1}\mid X_{0},\ldots,X_{n}] = Y_{n}$.
\end{definition}
Our martingale assumption is then formalized as follows:
\begin{assumption}[Martingale assumption]\label{ass:mar}
The stochastic process $\{\eta^{(i)}\}_{i=1}^\infty$ is a martingale with respect to the stochastic process $\{\mathcal{S}^{(i)}\}_{i=0}^\infty$. That is, for all $i \in \mathbb{N},$ $\E[\eta^{(i+1)}\mid\S^{(0:i)}] = \eta^{(i)},$ where $\S^{(0:i)}\vcentcolon=\{\S^{(0)},\ldots,\S^{(i)} \}$ denotes the random collection of subsets obtained by MRG up to iteration $i$.
\end{assumption}

One may question the justifiability of adopting such an assumption. An informal yet intuitive characterization of a martingale is that it is a stochastic process that does not demonstrate a predictable drift\cite{harris2022martingale}. In numerical experiments, this is indeed observed to be the case, rendering Assumption \ref{ass:mar} reasonable.

In the following, we use the next lemma to provide a concentration bound on the stochastic process $\{\eta^{(i)}\}_{i=1}^\infty$.
%
\begin{lemma}[Azuma-Hoeffding inequality~\cite{chung2006concentration}]\label{thm:mar}
Let $\{Y_{i}\}_{i=1}^\infty$ be a martingale with respect to $\{X_{i}\}_{i=1}^\infty$ such that $Y_1 = \mu$. If there exists a sequence of real numbers $\beta_i$ such that for all $i \in \mathbb{N}, \lvert Y_i-Y_{i-1}\rvert\leq \beta_i$, then for any $\lambda \in \mathbb{R}_{\geq 0}$,
\begin{equation}
    \Pr\left(Y_{n}-\mu \leq -\lambda \right) \leq \exp\left(-\frac{\textcolor{black}{2}\lambda^2}{\sum_{i=1}^n\beta_i^2}\right).
\end{equation}
\end{lemma}
Using Lemma \ref{thm:mar} and Assumption \ref{ass:mar}, we now study the performance of MRG for weak-submodular maximization subject to a budget constraint. We obtain the following theorem:
%
\begin{theorem}[Performance of MRG]\label{thm:mrg-mar}
Let $\{\eta^{(i)}\}_{i=1}^\infty$ satisfy Assumption \ref{ass:mar} and the conditions of Lemma \ref{thm:mar}, where, for all $i \in \mathbb{N},$ $\E[\eta^{(i)}]\geq \mu$, for some $\mu \in \mathbb{R}_{\geq 0}$. Then, for any confidence parameter $0<\delta <1$, MRG constructs a solution $\S_{\operatorname{mrg}}$ to \eqref{eq:fmax} such that
\begin{equation}\label{eq:thm:app-factor-mar}
    \frac{f(\S_{\operatorname{mrg}})}{f(\S^\ast)} \geq \frac{1- \exp\left[-\frac{1}{w_f}\left(\mu - \frac{c_{\max}}{B}\sqrt{\frac{U}{2}\log \frac{1}{\delta}}\right)\right]}{2w_{f}^{2}},
\end{equation}
with probability at least $1-\delta$, where $\S^\ast$ is an optimal solution, $U$ is the smallest integer such that $\sum_{j=1}^U \bar{c}_j \geq B$, and $c_{max} \vcentcolon= \Bar{c}_{|\Xs|}$.
\end{theorem}
\begin{proof}[Sketch of proof]
    The complete proof is stated in \cite{hibbard_hashemi_tanaka_topcu_2023}. The aim is to prove the result assuming that the selection returned by the inner loop of MRG in Algorithm \ref{alg:mg}  is limited to the subset $\S^{(i_v)}$ constructed up the to point before the condition in line 7 of the algorithm evaluates to false for the first time, since this will ensure that it will hold in the general case because of the monotone nondecreasing nature of $f$. 

    The subsequent applications of \eqref{eq:weak-sub}, \eqref{eq:prop_1}, and \eqref{eq:thm-mar:mrg:mar}, the use of induction over the iterations of the algorithm, and the invocation of Lemma \ref{thm:mar} allow us to place an upper bound on the ratio ${f(\S^{(i_v)})}/{f(\S^\ast)}$, completing the proof.
\end{proof}
One may wonder why lines 13 and 14 are present in the algorithm. The practical and intuitive reason is that these lines prevent certain adversarial scenarios where a single-element solution outperforms the greedily-constructed solution. More pragmatically, the full proof of Theorem \ref{thm:mrg-mar} uses these lines in the establishment of the performance lower-bound.
Theorem \ref{thm:mrg-mar} establishes a high-probability bound on the worst-case performance of MRG. Through observation of the right-hand side of \eqref{eq:thm:app-factor-mar}, one can see that the performance lower bound depends on several factors. Of these, $U$, $w_f$, $c_{\max}$ and $B$ are dictated by the problem itself and are not affected by the stochasticity of the algorithm. $\delta$ is the confidence parameter, which one sets \textit{a posteriori}, to see the performance lower bound guaranteed at a certain probability. 
Meaning, $\delta$ is also independent of the stochasticity of the algorithm. The only parameter which depends on the actual process of the algorithm is $\mu$, which one may think of as a lower bound on the expected values of all $\E[\eta^{(i)}]$ through the course of the algorithm. Even more intuitively, one can think of this value as a lower bound on the expected discrepancy between the performance of the standard \textsc{Greedy} procedure and MRG. When $c_{\max}\sqrt{U} \ll B$ 
(e.g., in the cardinality-constrained setting, where the cardinality bound $K$ satisfies $\sqrt{K} \ll K)$,
the right-hand side of \eqref{eq:thm:app-factor-mar} can be closely approximated by $(1-e^{-\mu/w_f})/{2w_f^{2}}$. 
Additionally, if $f$ is submodular (for which $w_f = 1$) and $\mu=1$, i.e., when $r_i = |\N|$ for all $i$, this last result reduces to $\frac{1}{2}(1-1/e)$, the approximation factor of \textsc{Greedy}~\cite{khuller1999budgeted}.
It is worth noting that the values of $c_{\max}$, $U$, and $B$ in equation \eqref{eq:thm:app-factor-mar} are either known beforehand or can be efficiently computed before running MRG. In contrast, $\omega_{f}$ and $\mu$ can be estimated through simulation. In certain specialized scenarios, such as minimizing the MSE in specific linear or quadratic observation models, upper bounds for $\omega_{f}$ can be found~\cite{hashemi2020randomized,hashemi2019submodular}.

The proof of Theorem \ref{thm:mrg-mar} also provides an explanation as to the previous determination of the size of the sets $\R^{(i)}$ at each step of MRG.
Considering that in the scenario with a cardinality constraint, the value of the cardinality bound $K$ matches the size of $\S^\ast$, and given that $ \lvert\S^\ast\rvert\leq U$, we suggest selecting $r_i = \frac{\lvert\N\rvert}{U}\log \frac{1}{\epsilon}$ sensors during each iteration of MRG. The results of $r_i$ on the runtime and the performance of the algorithm in a practical setting are demonstrated in Section \ref{sec:examples}, in Table \ref{tab:average_sol_time_mrg_kf} and Figure \ref{fig:various_MSE}.

It might be counterintuitive to the reader that the stochastic process $\eta^{(i)}$ appears to be unused in the algorithm. Put simply, $\eta^{(i)}$ describes a sequence of approximation ratios of the marginal-utility-per-cost performance of the stochastic greedy selection process to the marginal-utility-per-cost performance of the standard greedy selection process. Clearly, this is not a quantity that is meant to be (or \textit{can} be) used in the design of the algorithm, since one would have no access to realizations of this sequence unless one ran a copy of the standard greedy selection process at each iteration, completely annulling the computational efficiency of MRG and essentially ending up twice as costly. Rather, $\eta^{(i)}$, (along with $\mu$) is a natural byproduct resulting from the course of the algorithm, whose values are of no interest to us. The only quality we would like $\eta^{(i)}$ to possess is that it be a martingale, a reasonable assumption that appears to hold in relevant practical settings. 

Note that for $\delta = \exp{(-2/U(\mu B / c_{\max})^2)}$, the right-hand side is exactly equal to $0$, and for any value of $\delta$ less than or equal to this quantity, no guarantee is provided. 
In fact, for $\delta = \exp{(-2/U(\mu B / c_{\max})^2)}$, the right-hand side is exactly equal to $0$, and for any value of $\delta$ less than or equal to this quantity, no guarantee is provided. 


%

\begin{algorithm}[t]
	\caption{\textsc{Dual Randomized Greedy (DRG)}}
	\label{alg:dg}
\begin{algorithmic}[1]
		\STATE \textbf{Input:} Weak-submodular function $f$, ground set $\N$, performance threshold $A$, cost function $c$, sample set sizes $r_{i}$.
		\STATE \textbf{Output:} Subset $\S_{\operatorname{drg}}\subseteq \N$ with $f(\S_{\operatorname{drg}})\geq A$.
		\STATE Initialize $\S^{(0)} \gets \emptyset$, $\Xs \gets \N$, $i \gets 0$
		\WHILE{$f(\S^{(i)})< A$}
		\STATE Form $\R^{(i)}$ by sampling $\min \{r_i,\lvert\Xs\rvert\}$ elements from $\Xs$ uniformly at random.
		\STATE $j_s^{i} \in \argmax_{j \in \R^{(i)}} \frac{f_j(\S^{(i)})}{c_j}$
		\STATE $\S^{(i+1)} \gets \S^{(i)}  \cup \{j_s^{i} \}$
		\STATE $i\leftarrow i+1$
		\STATE $\Xs \leftarrow \Xs - \{j_s^{i} \}$
		\ENDWHILE
		\STATE $\S_{\operatorname{drg}} \gets \S^{(i)}$
        \STATE \textbf{return} $\S_{\operatorname{drg}}$
	\end{algorithmic}
\end{algorithm}
\subsection{Performance-Constrained Sensor Selection}
We now establish a high-probability bound for the performance of \textsc{Dual Randomized Greedy} (DRG) (Algorithm~\ref{alg:dg}) for the dual problem of \eqref{eq:fmax}, where we aim to satisfy a performance threshold while minimizing the total cost of the selected sensors. This problem is formally expressed as
\begin{equation}\label{eq:fmax:dual}
\begin{gathered}
\min_{\S \subseteq \N}c(\S)\\
\text{s.t.}\;  f(\S)\geq A,
\end{gathered}
\end{equation}
for some performance threshold $A \in \mathbb{R}_{\geq 0}$. 
Note that when $A = f(\N)$ and $f$ is submodular, \eqref{eq:fmax:dual} corresponds exactly to the well-known submodular set covering problem \cite{wolsey1982analysis}.
In that case, the standard \textsc{Greedy} algorithm, which again represents a special case of Algorithm \ref{alg:dg} in which $r_i = \lvert\N\rvert$, obtains the optimal approximation factor \cite{wolsey1982analysis}, given by
\begin{equation}\label{eq:dg-bound}
    \frac{c(\S_{\operatorname{g}})}{c(\S^\ast)} \leq 1+\log\left(\frac{f_j(\S^{(0)})}{f_j(\S^{(L-1)})}\right)\leq 1+\log\frac{ M}{m},
\end{equation}
where $L$ denotes the total number of iterations taken by \textsc{Greedy} until a solution is returned, $M \vcentcolon= \max_{j\in \N}f(\{j\})$, and $\quad m \vcentcolon= \min_{j\in \N \setminus \S^{(L-1)}}f_j(\S^{(L-1)}).$
Under the same observation made in~\eqref{eq:thm-mar:mrg:mar}, we establish that DRG
achieves a similar approximation factor with high probability.
%
\begin{theorem}[Performance of DRG]\label{thm:dual}
Let $\{\eta^{(i)}\}_{i=1}^\infty$ satisfy Assumption \ref{ass:mar} and the conditions of Lemma \ref{thm:mar}, where, for all $i \in \mathbb{N}, \E[\eta^{(i)}]\geq \mu$, for some $\mu \in \mathbb{R}_{\geq 0}$. Then, for any confidence parameter $0<\delta <1$, DRG finds a solution $\S_{\operatorname{drg}}$ to problem \eqref{eq:fmax:dual} such that, with probability at least $1-\delta$,
\begin{equation}\label{eq:thm:app-factor-dual}
\begin{aligned}
    \frac{c(\S_{\operatorname{drg}})}{c(\S^\ast)} &\leq \frac{{w_f}}{\mu} \left[1+(L-1)\log w_f+\log\frac{ M}{m}\right]+ \frac{1}{\mu c(\S^\ast)} \sqrt{\frac{1}{2}\log \frac{1}{\delta}c^2(\S_{\operatorname{drg}})},
    \end{aligned}
\end{equation}
where $\S^\ast$ is an optimal solution, $L \leq \lvert\N\rvert$ is the number of iterations required by DRG until termination, and $ c^2(\S_{\operatorname{drg}}) \vcentcolon= \sum_{j\in \S_{\operatorname{drg}}} c^2_j.$
\end{theorem}
\begin{proof}[Sketch of proof]
     The complete proof is stated in \cite{hibbard_hashemi_tanaka_topcu_2023}. Here, the fundamental idea is to consider a relaxation of \eqref{eq:fmax:dual}, which brings it into the domain of linear programming. The solutions to this relaxation can then be shown, by application of the weak-duality theorem, to be solutions to the original problem. 
     
     Lastly, Lemma \ref{thm:mar} is invoked to obtain the high-probability guarantee, completing the proof.
\end{proof}
Recall that the standard \textsc{Greedy} algorithm is a special case of DRG in which $r_i = \lvert\N\rvert$ for all $i$. In this case, we can set $\mu = \delta = 1$. Using the result of Theorem~\ref{thm:dual}, we directly obtain that
\begin{equation}
    \frac{c(\S_{\operatorname{drg}})}{c(\S^\ast)} \leq w_f \left[1+(L-1)\log w_f+\log\frac{ M}{m}\right],
\end{equation}
which extends the result of \cite{wolsey1982analysis} to the general setting of performance-constrained weak-submodular optimization that we have formulated in \eqref{eq:fmax:dual}, and generalizes the result established in \cite{JMLR:v19:16-534} for the uniform unit-cost case. This last result naturally reduces to \eqref{eq:dg-bound} for the case that $f$ is a submodular function, for which $w_{f} = 1$.
Much like the performance bound of MRG, the upper bound on the approximation ratio of DRG is dependent on terms that arise from the problem setting as well as those that are due to the stochasticity of the algorithm. In particular, $w_f$, $c(S^\ast)$ and $m$ are parameters that depend on the nature of the problem at hand and are independent of the algorithm. $\mu$ and $\delta$ serve very similar purposes to those in the case of MRG, and are already discussed thoroughly in that section. $M$ is interesting in that it depends both on the nature of the problem and the objective functions, and also the selections made by the algorithm up to the $(L-1)$th step. Finally, $L$ is a parameter that essentially indicates how long the algorithm runs until a solution is found, and depends entirely on the course of the algorithm.
Lastly, it's worth noting that the parameter $M$ can also be calculated efficiently before starting the DRG process, whereas the values of $L$ and $m$ are determined during the execution of the algorithm.
\section{A Randomized Method for Worst-Case Robustness}\label{sec:robust}
In this section, we study the application of the DRG algorithm to the solution of a standard formulation of the robust (weak) submodular maximization problem under the budget constraint proposed in \cite{RSOS}.
We assume that we are given a finite family of weak-submodular, monotone nondecreasing, and normalized functions $f^i, f^2, \ldots, f^n$.\footnote{The superscript notation is used to index the family of functions to avoid confusion with the marginal return notation.} We formulate the robust problem as follows:
\begin{equation}
\begin{gathered}
\max_{\S\subseteq\N}\min_{i\in[n]}f^i(\S) \label{robust}\\
\text{s.t.} \; c(\S) \leq B.
\end{gathered}
\end{equation}
We can visualize this problem as such: We have $n$ distinct performance objectives that we want to maximize by selecting the optimal subset $\S \subseteq \N$. Since we expect a satisfactory performance with respect to each of these objectives, we want to maximize the worst-performing one, all the while staying under a given budget constraint $B \in \mathbb{R}_{\ge 0}$.
In \cite{RSOS}, the \textsc{Submodular Saturation Algorithm} (SSA) is proposed, which solves a relaxation of \eqref{robust} in the setting where all of the objectives are submodular and a cardinality constraint is posed. SSA is presented in Algorithm \ref{alg:ssa}.
\begin{algorithm}[t]
\caption{\textsc{Submodular Saturation Algorithm} (SSA)\cite{RSOS}}
\label{alg:ssa}
\hspace*{\algorithmicindent}\textbf{Input:} Finite family of monotone nondecreasing submodular functions $f^1, \ldots f^n$, ground set $\N$, cardinality bound $K\le \lvert\N\rvert$, relaxation parameter $\alpha \geq 1$ \\
\hspace*{\algorithmicindent}\textbf{Output:} Solution set $\S$
\begin{algorithmic}[1]
\STATE $k_m \leftarrow 0$
\STATE $k_M \leftarrow \min_{i\in[n]} f^i(\N)$
\STATE $\S \leftarrow \emptyset$
\WHILE{$k_M - k_m \geq 1/n$}
\STATE $k \leftarrow (k_M - k_m)/2$
\STATE Define $\Bar{f}_{(k)}(\Xs) := \frac{1}{n}\sum_{i=1}^n \min\{f^i(\Xs), k\}$
\STATE $\hat{\S} \leftarrow \textsc{Greedy}(\Bar{f}_{(k)}, k)$
\IF{$\lvert\hat{\S}\rvert > \alpha K$}
\STATE $k_M \leftarrow k$
\ELSE
\STATE $k_m \leftarrow k$
\STATE $\S \leftarrow \hat{\S}$
\ENDIF
\ENDWHILE
\STATE \textbf{return} $\S$
\end{algorithmic}
\end{algorithm}

To facilitate the following discussion, we now provide a brief description of SSA. Given a finite family of normalized, monotone nondecreasing submodular functions $f^1, \ldots f^n$, a cardinality bound $K\in\mathbb{N}$, and a relaxation parameter $\alpha \geq 1$, SSA constructs a solution $\S \subseteq \N$ such that $f^i(\S) \geq k$ for a certain $k > 0$ and $\lvert\S\rvert \leq \alpha K$. Hence, it solves a relaxation of \eqref{robust}, where the budget constraint has been reduced to a cardinality constraint and relaxed by a factor of $\alpha$. 

SSA leverages two useful facts. The first is that when $f$ is a submodular function, so is the truncated function given by $f_{(k)}(\Xs) \vcentcolon= \min\{f^i(\Xs), k\}$ \cite{RSOS}, which we call $f$ truncated at $k$. The second is that the nonnegative weighted sum (and hence the mean) of submodular functions is again submodular \cite{RSOS}. The algorithm then truncates the input functions $f^1, \ldots, f^n$ and averages them, producing the submodular function $\Bar{f}_{(k)}$. It adaptively decides on the truncation value $k$ by doing a line search over the feasible values of $k$, and outputs a solution $\S \subseteq \N$ with $f^i(\S) \geq k$ for all $i \in [n]$, where $k$ is the highest value found by the line search procedure. As can be seen in line 7, SSA makes use of the \textsc{Greedy} as a subroutine. The aim of this section is twofold: to extend the result given in \cite{RSOS} using \textsc{Greedy} to one using the computationally cheaper, randomized version of DRG, and to further generalize the result to the case involving weak submodular functions and a more complex budget constraint with nonuniform costs instead of the simpler cardinality-constrained setting.

Now, with regard to SSA, we have the following theoretical guarantee:
\begin{theorem}[Performance of SSA\cite{RSOS}]\label{thm:ssa}
For any cardinality bound $K \in \mathbb{N}$, SSA finds a solution $\hat{\S}$ to a relaxation of \eqref{robust}, where $f^i$ are submodular for all $i \in [n]$, such that
\begin{equation}
\min_{i\in[n]} f^i(\hat{\S}) > \max_{\S\subseteq \N, c(\S)\leq B}\min_{i\in[n]} f^i(\S)
\end{equation}
and $c(\hat{\S}) \leq \alpha B,$ for $\alpha = 1 + \log(\max_{j\in \N} \sum_{i=1}^n f^i(\{j\})$.
\end{theorem}

The above result relies on the following lemma concerning the \textsc{Greedy} algorithm \cite{wolsey}:

\begin{lemma}\label{lemma:gpc}
Given a submodular function $f$ over a ground set $\N$, and a feasible performance threshold value $k$, the \textsc{Greedy} algorithm produces a solution $\hat{\S}$ such that
\begin{equation}
    \dfrac{\lvert\hat{S}\rvert}{\lvert \S^\ast\rvert} \leq 1 + \log\max_{j\in \N}f(\{j\}),
\end{equation}
where $\S^\ast$ is an optimal solution (i.e., $\S^\ast = \argmin\{\lvert\S\rvert: f(\S) \geq k\}$).
\end{lemma}

With this result, it becomes clear how Theorem \ref{thm:ssa} provides this guarantee: since \textsc{Greedy} is guaranteed to return a solution $\hat{\S}$ such that $\lvert\hat{\S}\rvert/\lvert\S^\ast\rvert \leq 1 + \log\max_{j\in \N}f(\{j\})$, we may set our relaxation parameter $\alpha$ in SSA to exactly this upper bound on the right-hand side. Doing so, we know that if the \textsc{Greedy} subroutine fails to return a solution for a given performance threshold $k$, then that value of $k$ must be infeasible. Consequently, it only remains to find the highest $k$ value that \textit{is} feasible, which is achieved by a simple line search procedure. For the line search, the values for the interval to be searched at the current iteration are set in lines 1 and 2 and are updated after each run of the \textsc{Greedy} subroutine. If the \textsc{Greedy} subroutine fails to return a solution, we know that the most recent $k$ value provided was infeasible, so we update the higher end of our search interval to that value (line 9). If, on the other hand, the \textsc{Greedy} subroutine did return a solution, then the most recent $k$ value provided was feasible. This case does not preclude the possibility that there is yet another, higher $k$ value which is feasible, so we continue the line search by updating the lower end of our search interval to this value (line 11).

It may not be readily apparent how giving the average of the truncated $f^1, \ldots, f^n$, i.e., $\Bar{f}_{(k)}$ as input into the \textsc{Greedy} subroutine guarantees that the solution $\hat{\S}$ returned by \textsc{Greedy} satisfies $f^1(\hat{\S}), \ldots, f^n(\hat{\S}) \geq k$. It suffices to see that the average of these functions $f^1(\hat{\S}), \ldots, f^n(\hat{\S})$ truncated at $k$ is equal to $k$ if and only if $f^1(\hat{\S}), \ldots, f^n(\hat{\S}) \geq k$. In other words, even if a single $f^j(\hat{\S}) < k$, this would make its truncated counterpart $f^j_{(k)}(\hat{S}) < k,$ and hence even if for all other $i \neq j, f^i(\hat{\S}) \geq k$ holds, their average would not satisfy the performance threshold, making $\Bar{f}_{(k)} = \frac{1}{n}\sum_{i=1}^n f^i_{(k)}(\hat{\S}) < k$.

Before directly analyzing the use of DRG instead of \textsc{Greedy} as a subroutine in SSA, there yet remains a gap to be bridged between the setting of \cite{RSOS} and this work, that is, the consideration of weak-submodular functions instead of proper submodular functions. To derive a result similar to Theorem \ref{thm:ssa} in the case of weak-submodular functions, we need to demonstrate that averaging and truncation preserve weak submodularity as it preserves submodularity. This is indeed the case, as shown in the following two lemmas.

\begin{lemma}\label{thm:ws-sum}
Let $f^1, \ldots, f^n$ be weak submodular with WSCs $w_{f^1}, \ldots, w_{f^n} < \infty.$
Any nonnegative weighted sum $\Bar{f}$ of $f^1, \ldots f^n$ is also weak submodular, with $w_{\Bar{f}} < \infty$.
\end{lemma}
In particular, note that this result implies that the mean of a finite number of weak-submodular functions is yet again a weak-submodular function. We now show a similar result for the truncation of a weak-submodular function.

\begin{lemma}\label{thm:ws-truncation}
Let $f$ be a weak-submodular function, i.e., $w_f < \infty.$ Then, the truncation of $f$ at $k \leq 0$, which is denoted $f_{(k)}(\S) \vcentcolon= \min\{f(\S), k\},$ with the convention that if $f_{(k)_i}(\T)/f_{(k)_i}(\S) = {0}/{0},$ then $f_{(k)_i}(\T)/f_{(k)_i}(\S) = 0,$ is a weak-submodular function with $w_{f_{(k)}} < \infty$.
\end{lemma}
\begin{proof}
By our definition of weak submodularity, any set function that is monotone nondecreasing and has a bounded WSC is weak submodular, so we first show that $f_{(k)}$ is monotone nondecreasing. For any pair of sets $\S \subseteq \T$ where $f(\S), f(\T) \leq k$, $f_{(k)}(\S) = f(\S) \leq f(\T) = f_{(k)}(\T)$. For any pair of sets $\S \subset \T$ where $f(\S) \leq k$ and $f(\T) > k$, $f_{(k)}(\S) = f(\S) \leq k = f_{(k)}(\T).$ Finally, for any pair of sets $\S \subseteq \T$ where $f(\S), f(\T) > k$, $f_{(k)}(\S) = k = f_{(k)}(\T)$. Hence, $f_{(k)}$ is monotone nondecreasing.

Now, let us recall the following definition: 
\begin{align}
w_{f_{(k)}} \vcentcolon= \max_{(\S, \T, i)\in \Tilde{\N}} f_{(k)_i}(\T) / f_{(k)_i}(\S).
\end{align}
Say that this maximum is attained at $(\S_0, \T_0, i_0).$ Due to the monotone nondecreasing nature of $f_{(k)},$ we have that
\begin{equation}
f_{(k)}(\S_0) \leq f_{(k)}(\T_0) \leq f_{(k)}(\T_0 \cup \{i_0\}),
\end{equation}
and
\begin{equation}
f_{(k)}(\S_0) \leq f_{(k)}(\S_0 \cup \{i_0\}) \leq f_{(k)}(\T_0 \cup \{i_0\}). 
\end{equation}
Let us say that a truncated function $f_{(k)}$ is \textit{saturated} at $\S$ if $f_{(k)}(\S) = k$. In light of these inequalities, the remainder of the proof reduces 
to checking whether or not $f_{(k)}$ is saturated at each of $\S_0, \T_0, \S_0 \cup \{i_0\},$ and $\T_0 \cup \{i_0\}$. For instance, if $f_{(k)}$ is saturated at $\S_0,$ we have, by convention, $w_{f_{(k)}} = {0}/{0} = 0 < \infty.$ 
If, on the other hand, $f_{(k)}$ is not saturated at $\S_0$ but at $\T_0,$ we have, $w_{f_{(k)}} = 0 < \infty.$ Investigating the rest of the valid configurations, we obtain
\begin{equation}
w_{f_{(k)}} \in \mathcal{W} \vcentcolon= \left\{0, \frac{k - f(\T_0)}{k - f(\S_0)}, \frac{k - f(\T_0)}{f_{{(k)}_{i_0}}(\S_0)}, w_f\right\}.
\end{equation}
It is important to note that $f_{(k)_{i_0}}(\S_0) = f_{(k)}(\S_0 \cup \{i_0\}) - f_{(k)}(\S_0)$ is obtained only in the case where $f_{(k)}$ is not saturated at $\S_0 \cup \{i_0\}$ (and hence at $\S_0$), which means that  $f_{(k)}(\S_0 \cup \{i_0\}) = f(\S_0 \cup \{i_0\})$ and $f_{(k)}(\S_0) = f(\S_0).$ Hence, it must be that $f_{(k)}(\S_0 \cup \{i_0\}) \neq f_{(k)}(\S_0),$ as otherwise, $(\S_0, \T_0, i_0)$ would have been a maximizer of $w_{f} \vcentcolon= {\max_{(\S,\T,i)\in \tilde{\Xs}}{f_i(\T)\slash f_i(\S)}},$ making $w_f = \infty,$ violating the premise that $w_f < \infty.$ We conclude, then, that every element of $\mathcal{W}$ is a finite value, i.e., $w_{f_{(k)}} < \infty.$
\end{proof}
With these two results, we are now ready to present \textsc{Randomized Weak Submodular Saturation Algorithm (Random-WSSA)}. To summarize, \textsc{Random-WSSA} uses the main body of SSA, with the \textsc{Greedy} subroutine replaced by its randomized variant DRG, and furthermore, works in the presence of weak-submodular functions. The full algorithm is presented in Algorithm \ref{alg:random-wssa}.

\begin{algorithm}[t]
\caption{\textsc{Randomized Weak Submodular Saturation Algorithm (Random-WSSA)}}
\label{alg:random-wssa}
\hspace*{\algorithmicindent}\textbf{Input:} Finite family of monotone nondecreasing weak-submodular functions $f^1, \ldots f^n$, ground set $\N$, budget bound $B\in\mathbb{R}_{\ge 0}$, cost function $c$, relaxation parameter $\alpha \geq 1$ \\
\hspace*{\algorithmicindent}\textbf{Output:} Solution set $\S$
\begin{algorithmic}[1]
\STATE $k_m \leftarrow 0$
\STATE $k_M \leftarrow min_{i\in[n]} f^i(\N)$
\STATE $\S \leftarrow \emptyset$
\WHILE{$k_M - k_m \geq 1/n$}
\STATE $k \leftarrow (k_M - k_m)/2$
\STATE Define $\Bar{f}_{(k)}(\S) := \frac{1}{n}\sum_{i=1}^n \min\{f^i(\S), k\}$
\STATE $\hat{\S} \leftarrow \operatorname{DRG}(\Bar{f}_{(k)}, k)$
\IF{$c(\hat{\S}) > \alpha B$}
\STATE $k_M \leftarrow k$
\ELSE
\STATE $k_m \leftarrow k$
\STATE $\S \leftarrow \hat{\S}$
\ENDIF
\ENDWHILE
\STATE \textbf{return} $\S$
\end{algorithmic}
\end{algorithm}

From the preceding results, it is straightforward to show how the theoretical guarantee for \textsc{Random-WSSA} is established. We formalize this result in the following theorem.
\begin{theorem}\label{thm:random-wssa}
Suppose the conditions of Theorem \ref{thm:dual} hold. Then, for any confidence parameter $0<\delta<1$, and for any budget constraint $B \in \mathbb{R}_{\geq0}$, \textsc{Random-WSSA} finds a solution $\hat{\S}$ to a relaxation of \eqref{robust}, where $f^i$ are weak submodular, i.e., are monotone nondecreasing with WSCs $w_{f^i} < \infty$ for all $i \in [n]$, such that with probability at least $(1-\delta)^P$,
\begin{equation}
\min_{i\in[n]} f^i(\hat{\S}) > \max_{c(\S)\leq B}\min_{i\in[n]} f^i(\S)
\end{equation}
and $c(\hat{\S}) \leq \alpha B,$ for 
\begin{equation}
\begin{split}
\alpha &= \dfrac{w_{\Bar{f}_{(k)}}}{\mu}\left[1+(L-1)\log(w_{\Bar{f}_{(k)}}) + \log \dfrac{M}{m}\right] + \frac{1}{\mu c(\S^\ast)}\sqrt{\frac{1}{2}\log\frac{1}{\delta}c^2(\hat{\S})},
\end{split}
\end{equation}
where $L \leq \lvert\N\rvert$ is the maximum number of iterations required by DRG until termination throughout all iterations of \textsc{Random-WSSA}, $P$ is the number of iterations required by \textsc{Random-WSSA} until termination, $c^2(\S) := \sum_{j\in \S} c_j^2, M = \max_{j\in \N, i\in[n]} f^i(\{j\}),$ and $m = \min_{j\in \N, i\in[n]} f^i_j(\N \setminus \{j\}).$
\end{theorem}
\begin{proof}
The proof follows immediately from Theorem \ref{thm:dual}, and Lemmas \ref{thm:mar}, \ref{thm:ws-sum}, \ref{thm:ws-truncation}. Note that since the DRG subroutine is run $P$ times throughout the execution of \textsc{Random-WSSA}, and that these runs are statistically independent random events, the high-probability guarantee $(1-\delta)$ is raised to the exponent $P$, to account for each of the $P$ times that DRG must succeed in returning a solution.
\end{proof}
To provide concrete values for the setting of $\alpha$, note that the upper bound in Theorem \ref{thm:dual} is given in terms of the quantities $c^2(\hat{\S})$ and $c(\S^*)$. One may upper-bound the first term instead by $c^2(\N)$ and lower-bound the second term instead by $1$. In case $c(\S^\ast) < 1$, one must only scale the cost value $c_j$ of every element $j\in \N$ and the budget $B$ by the factor $1/c_{\min}$, where $c_{\min} = \min_{j\in \N} c_j.$ Doing so ensures that for all $\S \subseteq \N, c(\S) \geq 1$, and in particular, $c(\S^\ast) \geq 1$, after which we can safely do the lower-bounding by $1$. However, in practical settings, one usually does not have to resort to such strategies, as experimental results demonstrate that \textsc{Random-WSSA} very frequently succeeds to construct solutions for the non-relaxed version of the problem, i.e., with $\alpha=1$.


\section{Numerical Examples}\label{sec:examples}

\begin{figure}[t]
     \centering
     \includegraphics[trim={2.0cm 2.75cm 1.5cm 3.0cm},clip,width=0.75\columnwidth]{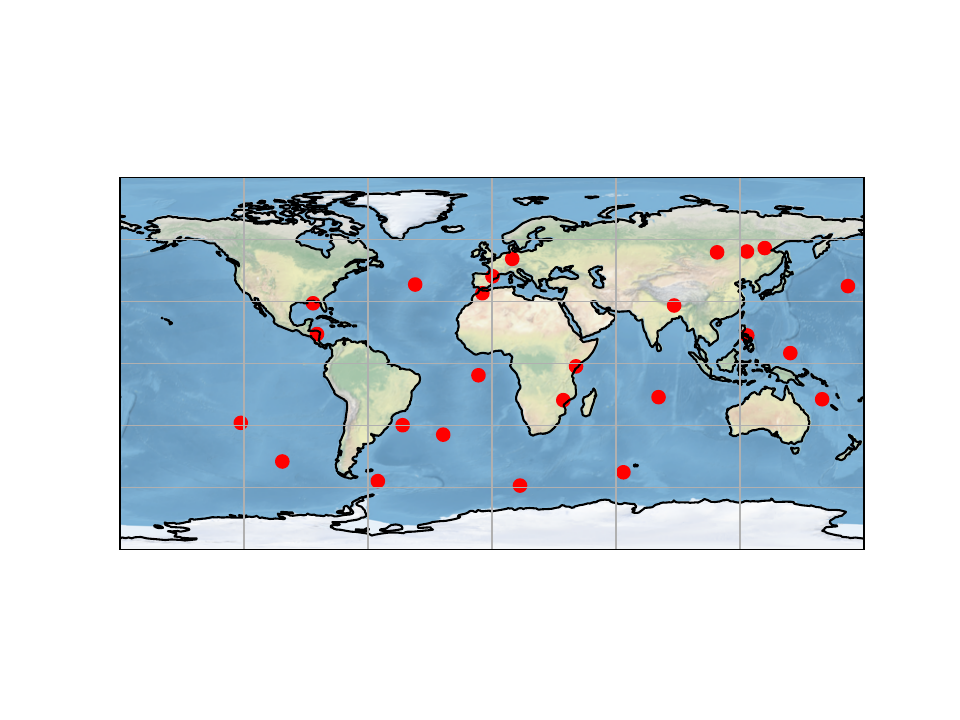}
    \caption{Locations of the $25$ randomly-selected atmospheric points of interest.}
    \label{fig:atm_points_of_interest}
\end{figure}
\begin{table}[t]
    \centering
    \caption{Average MRG solution time in seconds over the simulation horizon for various $r_{i}$ and budgets $B$. Note that $r_i =240$ corresponds to using the entire ground set, i.e., the standard \textsc{Greedy} algorithm.}
    \begin{NiceTabular}{|c||c|c|c|c|}
        \hline
        ${ }$ & $B = 25$ & $B = 50$ & $B = 75$ & $B = 100$ \\ \hline \hline
        $r_{i} = 60$ & $1.52$ & $1.59$ & $1.59$ & $1.58$  \\ \hline
        $r_{i} = 120$ & $2.93$ & $2.95$ & $3.03$ & $3.01$ \\ \hline
        $r_{i} = 180$ & $3.96$ & $3.96$ & $3.98$ & $3.96$ \\ \hline
        $r_{i} = 240$ & $4.33$ & $4.33$ & $4.39$ & $4.34$ \\ \hline
    \end{NiceTabular}
    \label{tab:average_sol_time_mrg_kf}
\end{table}
\begin{figure}
     \centering
     \begin{subfigure}[b]{\columnwidth}
         \centering
         \includegraphics[width=.9\textwidth]{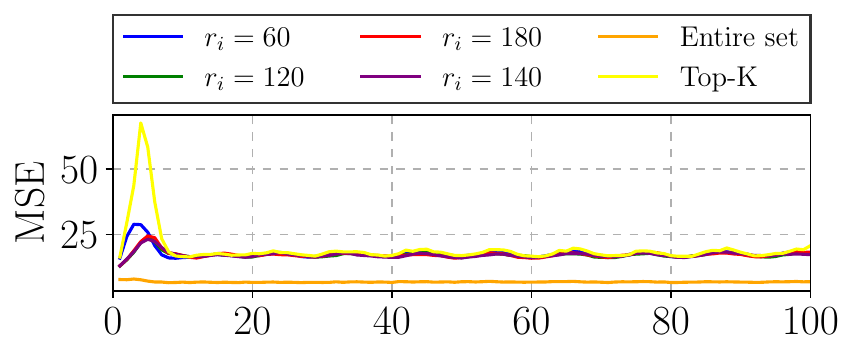}
     \end{subfigure}
     \begin{subfigure}[b]{\columnwidth}
         \centering
         \includegraphics[width=.9\textwidth]{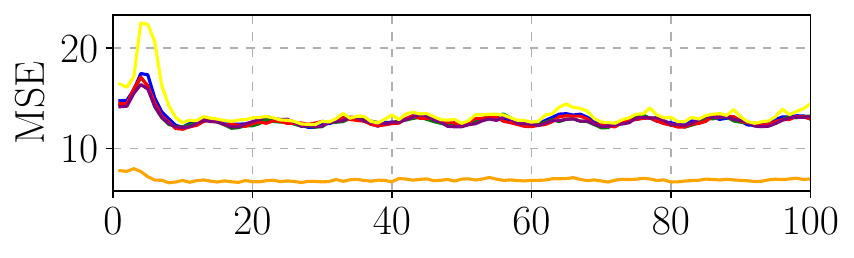}
     \end{subfigure}
     \begin{subfigure}[b]{\columnwidth}
         \centering
         \includegraphics[width=.9\textwidth]{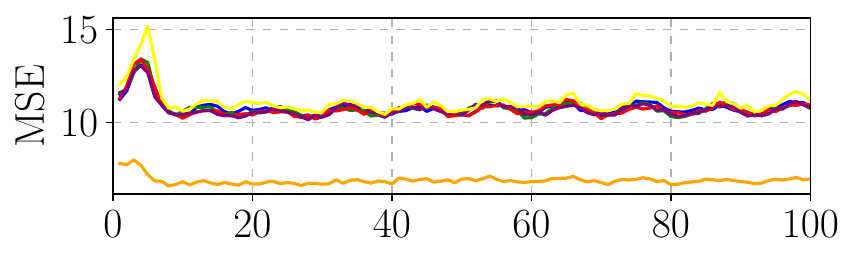}
     \end{subfigure}
     \begin{subfigure}[b]{\columnwidth}
         \centering
         \includegraphics[width=.9\textwidth]{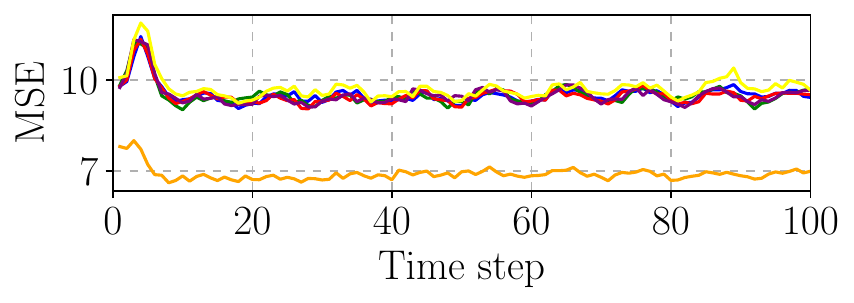}
     \end{subfigure}
    \caption{Total mean-square errors (MSE) over time for various $r_{i}$ and $B$. From top to bottom are the plots for $B=25$, $50$, $75$, and $100$, respectively. Performance of using the entire set of satellites as well as that of picking the highest elements in terms of marginal-gain-to-cost ratio with respect to the empty set until the budget constraint is met (which we call Top-$K$) are included as benchmark values.}
    \label{fig:various_MSE}
\end{figure}

\begin{table}[t]
    \centering
    \caption{Average MRG MSE values incurred over the simulation horizon for various $r_{i}$ and budgets $B$. Average MSE values for using the entire set and utilizing a Top-K approach are included as benchmarks. Note that $r_i =240$ corresponds to using the entire ground set, i.e., the standard \textsc{Greedy} algorithm.}  
    \begin{NiceTabular}{|c||c|c|c|c|}
        \hline
        ${ }$ & $B = 25$ & $B = 50$ & $B = 75$ & $B = 100$ \\ \hline \hline
        $r_{i} = 60$ & $17.73$ & $12.93$ & $10.78$ & $9.52$  \\ \hline
        $r_{i} = 120$ & $17.27$ & $12.82$ & $10.78$ & $9.52$ \\ \hline
        $r_{i} = 180$ & $17.33$ & $12.83$ & $10.75$ & $9.50$ \\ \hline
        $r_{i} = 240$ & $17.23$ & $12.81$ & $10.74$ & $9.51$ \\ \hline
        Entire set & $6.89$ & $6.89$ & $6.89$ & $6.89$
        \\ \hline
        Top-$K$ & $19.39$ & $13.52$ & $11.12$ & $9.74$
        \\ \hline
    \end{NiceTabular}
    \label{tab:average_mse_mrg}
\end{table}

We now corroborate our theoretical results on the proposed algorithms for tasks involving sensor selection in the context of LEO satellite constellations.
We consider a Walker-Delta constellation parameterized by $i:T/P/f$, where $i$ is the orbit inclination, $T$ is the number of satellites, $P$ is the number of unique orbits, and $f$ is the relative spacing between the satellites~\cite{walker1984satellite}.
The constellation altitude is $2000$ km, and the satellites are assumed to remain Earth-pointing with conical field-of-views of angle ${\pi}/{3}$.


\subsection{Sensor Selection for Atmospheric Monitoring with MRG}\label{subsec:kf_example}

Consider a LEO satellite constellation parameterized by $60^\circ:240/12/1$ performing atmospheric sensing over a set of $25$ randomly-instantiated points of interest, as shown in Figure~\ref{fig:atm_points_of_interest}.
We simulate the conditions at each point $\mathbf{x}_{p}(t)=[x_{p}(t),y_{p}(t),z_{p}(t)]^{\top}$ with the Lorenz 63 model~\cite{lorenz1963deterministic}, a simplified model of atmospheric convection, in which $x_{p}$ is proportional to the rate of convection while $y_{p}$ and $z_{p}$ are proportional to the horizontal and vertical temperature variation, respectively.
Precisely, the dynamics at each point are described by
\begin{align*}
    \dot{\mathbf{x}}_{p} = \begin{bmatrix} \dot{x}_{p} \\ \dot{y}_{p} \\ \dot{z}_{p} \end{bmatrix} = \kappa \begin{bmatrix} \sigma (y_{p}- x_{p}) \\ x_{p}(\rho - z_{p}) - y_{p} \\ x_{p} y_{p} - \beta z_{p} \end{bmatrix} + \omega_{p},
\end{align*}
where $\kappa, \sigma, \rho, \beta \in \mathbb{R}$. Note that we have omitted the dependence on time $t$ for notational clarity.
For this simulation, we choose values of $\kappa=0.005, \sigma = 10$, $\beta = \frac{8}{3}$, and $\rho = 28$. With this selection of parameters, the system exhibits chaotic behavior.
Furthermore, $\omega_{p} \sim \mathcal{N}([0,0,0]^{\top},\Sigma_{\omega_{p}})$ is a zero-mean Gaussian process modeling the noise. For this simulation, we set $\Sigma_{\omega_{p}} = 0.1I_{3}$.

We assign the $n$th satellite a cost $c_{n}$ uniformly at random from the interval $[1,2]$.
At each time step of the simulation, the decision-maker aims to make an optimal selection of a subset of the satellites to minimize the MSE of the estimated atmospheric reading, which is known to be a weak-submodular function~\cite{hashemi2017randomized}, while staying under a given budget constraint $B$.
In our setting, each time step of the simulation corresponds to $\Delta t = 60$ seconds of time in the simulation.
The budget constraint $B$ itself could practically be used to model several real-life limitations, e.g., communication costs between the satellites and the decision-maker or the operational cost of obtaining a reading from a specific satellite. In a real scenario, the cost values of each satellite should be tailored to the specific mission parameters.

The decision-maker uses an unscented Kalman filter~\cite{bar2004estimation} to estimate the atmospheric state at each point of interest, using the measurement model
\begin{align}\label{eq:observation_model}
    \mathbf{z}_{p,n}^{s} = \mathbf{1}_{p,n}^{s} \mathbf{x}_{p}(s \Delta t) + \nu_{p,n}^{s},
\end{align}
where $\mathbf{z}_{p,n}^{s} \in \mathbb{R}^{3}$ is the $n^{th}$ satellite's observation of point $p$ at step $s$, $\mathbf{1}_{p,n}^{s}$ is an indicator function equal to $1$ if $p$ is in the field-of-view of the $n^{th}$ satellite at step $s$ and $0$ otherwise, and $\nu_{p,n}^{s} \sim \N([0,0,0]^{\top},\Sigma_{\nu_{p}})$ is Gaussian measurement noise with  $\Sigma_{\nu_{p}} = 2 I_{3}$.
As demonstrated in~\cite{hashemi2017randomized}, the objective of minimizing the MSE of the estimator obtained from the selection of satellites can be directly modeled as maximizing a monotone nondecreasing weak-submodular performance objective.


We test several sampling set sizes $r_{i}$ and budgets $B$, simulating the estimation task over a horizon of $100$ time steps. At each time step, we update the satellite positions and atmospheric conditions. Subsequently, we run MRG at each time step to reconstruct a feasible selection of satellites to use for atmospheric reading. As benchmarks, we include two additional methods of selecting satellites. One consists of not making a selection at all and using the entire constellation of satellites at all times. This demonstrates the ideal scenario of the lowest achievable MSE value at each step, however, it also always incurs the highest cost possible, violating any meaningful constraint on the budget. Additionally, we include the method which we denote by Top-$K$. This entails sorting the entire set of satellites in decreasing order of marginal-gain-to-cost ratio and then sequentially adding them to the selection as long as they do not violate the budget constraint. 
%
%
%
Figure~\ref{fig:various_MSE} shows the resulting MSE throughout the simulation for different combinations of $r_{i}$ and $B$ values.
Table~\ref{tab:average_sol_time_mrg_kf} shows the average wall-clock computation time of MRG for each combination. Table~\ref{tab:average_mse_mrg} shows the average MSE values  over the entire time horizon of the simulation incurred by different methods for varying values of $r_{i}$ and $B$, effectively displaying the results in Figure~\ref{fig:various_MSE} in tabulated format.
As would be expected, increasing the budget value $B$ allows MRG to add additional satellites whose field-of-view may contain additional points of interest, increasing the total utility achieved, and hence, reducing the incurred MSE.
It is interesting to note, however, that although increasing the size $r_i$ of the sampling set at each iteration significantly increases the wall-clock computation time needed to construct solutions, the performance in terms of minimizing the MSE appears to be almost independent of this parameter. This effect is especially pronounced as the budget bound $B$ increases, allowing the randomized selection process more tolerance over the long run. Put another way, as the budget $B$ increases, we implicitly allow the selection process more lenience to make suboptimal selections and then to correct these with better selections with the more ample budget that still remains. This explains the quasi-independence of the performance of MRG from the size of the sampling set $r_i$, in high-budget scenarios.
Recall that the case where $r_{i}=240$ corresponds to the standard \textsc{Greedy} algorithm. It is clear to see the computational advantage of MRG over the full \textsc{Greedy} algorithm from Table \ref{tab:average_sol_time_mrg_kf}.
All of these observations demonstrate that, empirically, randomization significantly reduces algorithmic runtime while incurring only a modest loss in terms of performance. Furthermore, the more sophisticated approach of using a \textsc{Greedy} algorithm is also justified against using the simpler approach of Top-$K$, judging by the decrease in MSE values brought about by the introduction of greedy selection. This effect is especially prevalent in lower-budget regimes.


\begin{table}[t]
    \centering
    \caption{Average budget cost incurred by the satellite selection over the simulation horizon for various $r_{i}$ and coverage fractions $F$. Note that $r_i =240$ corresponds to using the entire ground set, i.e., the standard \textsc{Greedy} algorithm.}
    \begin{NiceTabular}{|c||c|c|c|}
        \hline
        ${ }$ & $F = 0.5$ & $F = 0.7$ & $F = 0.9$ \\ \hline \hline
        $r_{i} = 60$ & $63.08$ & $99.29$ & $161.74$ \\ \hline
        $r_{i} = 120$ & $62.64$ & $99.30$ & $161.47 $ \\ \hline
        $r_{i} = 180$ & $62.69$ & $99.24$ & $161.34$ \\ \hline
        $r_{i} = 240$ & $62.89$ & $99.10$ & $161.24$ \\ \hline
    \end{NiceTabular}
    \label{tab:average_budget_drg_eo}
\end{table}
\begin{table}
    \centering
    \caption{Average DRG solution time in seconds (Algorithm~\ref{alg:dg}) over the simulation horizon for various $r_{i}$ and coverage fractions $F$. Note that $r_i =240$ corresponds to using the entire ground set, i.e., the standard \textsc{Greedy} algorithm.}
    \begin{NiceTabular}{|c||c|c|c|}
        \hline
        ${ }$ & $F = 0.5$ & $F = 0.7$ & $F = 0.9$ \\ \hline \hline
        $r_{i} = 60$ & $0.53$ & $0.86$ & $1.37$ \\ \hline
        $r_{i} = 120$ & $1.08$ & $1.67$ & $2.64$ \\ \hline
        $r_{i} = 180$ & $1.60$ & $2.50$ & $3.64$ \\ \hline
        $r_{i} = 240$ & $2.00$ & $2.80$ & $3.96$ \\ \hline
    \end{NiceTabular}
    \label{tab:average_sol_time_drg_eo}
\end{table}
\begin{figure}
     \centering
     \begin{subfigure}[b]{\columnwidth}
         \centering
         \includegraphics[trim={2.0cm 2.75cm 1.5cm 3.0cm},clip,width=0.75\textwidth]{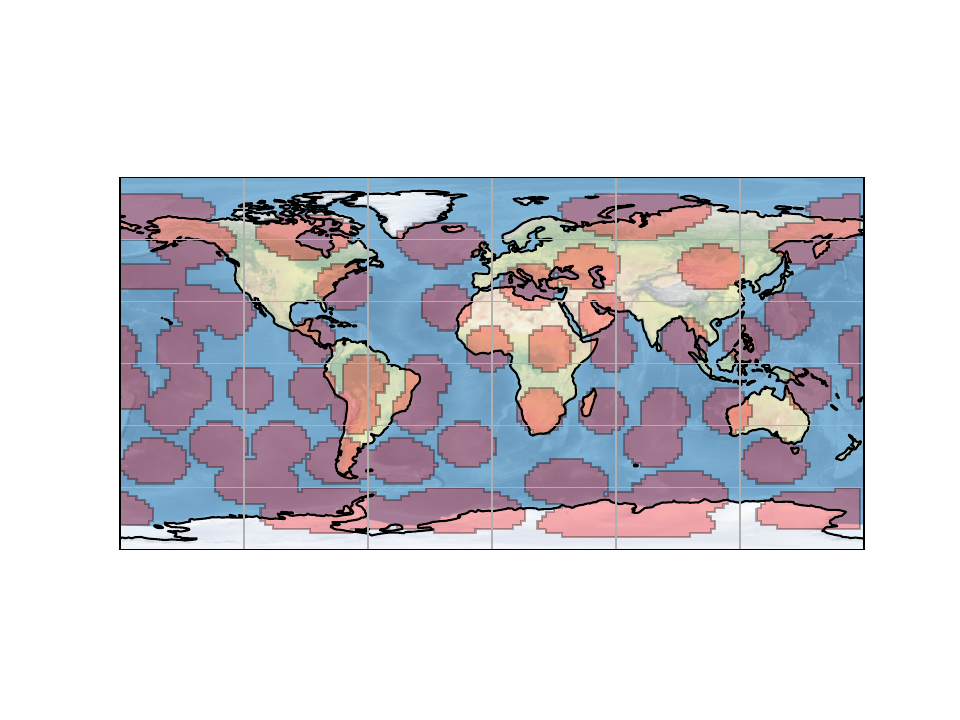}
         \caption{Ground coverage obtained for $r_{i}=120$ and $F=0.5$, time step $50$.}
         \label{fig:ri_120_CF_50_tstep_50}
     \end{subfigure}
     \begin{subfigure}[b]{\columnwidth}
         \centering
         \includegraphics[trim={2.0cm 2.75cm 1.5cm 3.0cm},clip,width=0.75\textwidth]{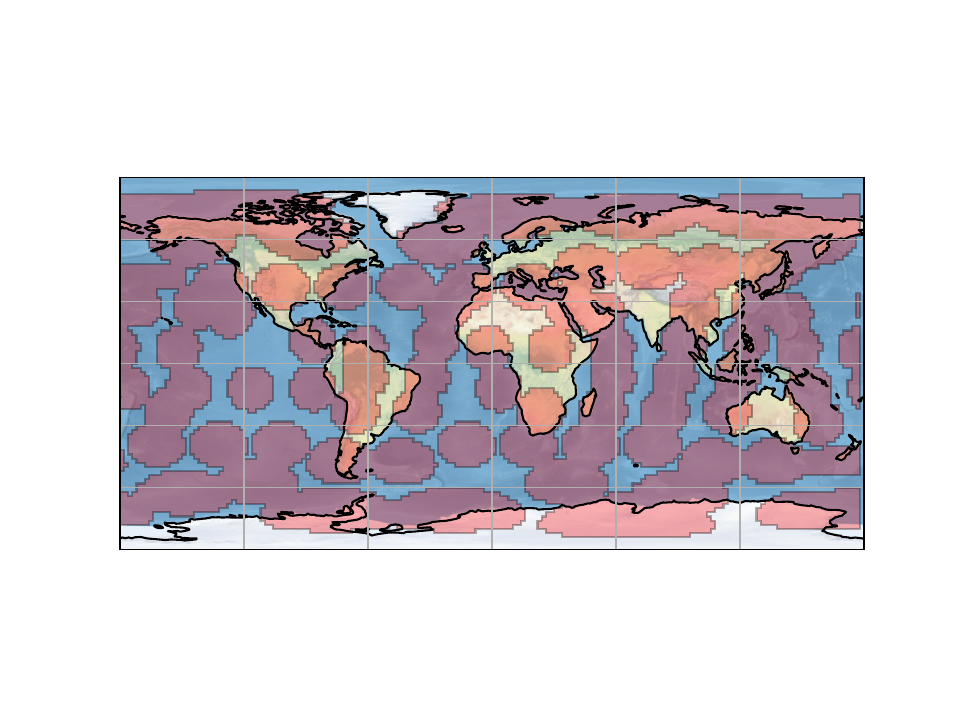}
         \caption{Ground coverage obtained for $r_{i}=120$ and $F=0.7$, time step $50$.}
         \label{fig:ri_120_CF_70_tstep_50}
     \end{subfigure}
     \begin{subfigure}[b]{\columnwidth}
         \centering
         \includegraphics[trim={2.0cm 2.75cm 1.5cm 3.0cm},clip,width=0.75\textwidth]{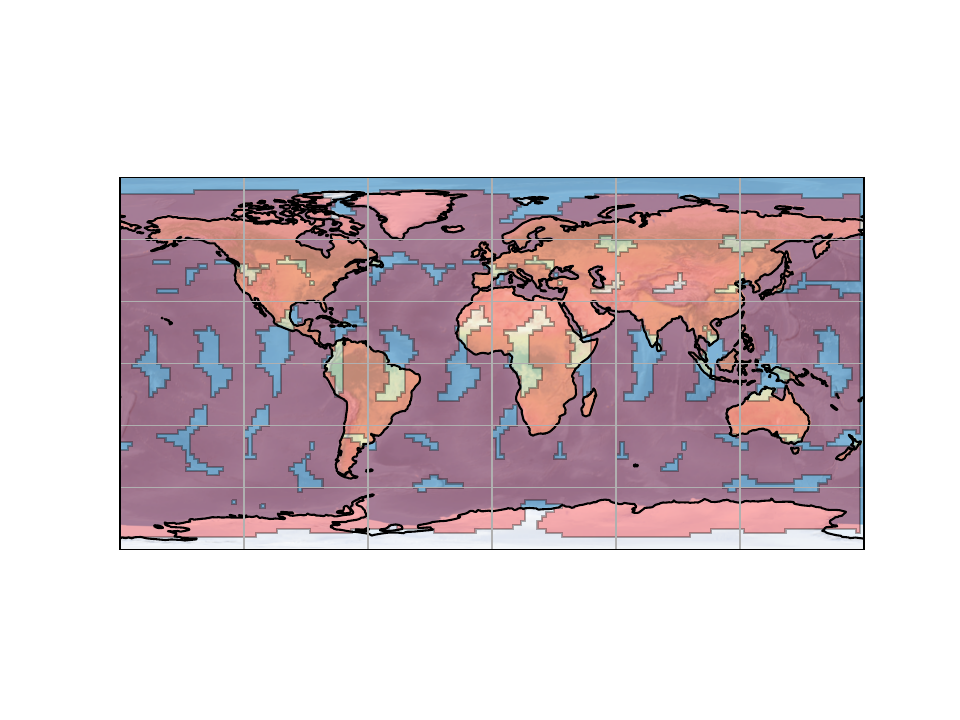}
         \caption{Ground coverage obtained for $r_{i}=120$ and $F=0.9$, time step $50$.}
         \label{fig:ri_120_CF_90_tstep_50}
     \end{subfigure}
        \caption{Visualization of the ground coverage areas obtained at time step $50$ through DRG using $r_{i}=120$ and various $F$.}
        \label{fig:various_ground_coverage}
\end{figure}

\subsection{Minimum Ground Coverage with DRG}\label{subsec:coverage quality}

We now consider a scenario where the goal of the decision-maker is to select a subset of satellites whose coverage area is at least a given fraction of the maximum achievable coverage area, which we denote by $F$ while minimizing the incurred cost of the selection.
This objective can naturally be expressed in terms of the performance-constrained problem~\eqref{eq:fmax:dual}.
For this simulation, we keep the same constellation parameters as in Section \ref{subsec:kf_example}, with the exception of adjusting the inclination of the constellation to $75^\circ$.

To calculate the area covered by a satellite, we first discretize Earth's surface into a grid of cells of width and height $2^\circ$.
The area of each cell is then calculated assuming a spherical Earth.
At each time step of the simulation, a cell's area is said to be covered by a satellite if the centroid of that cell happens to be contained inside the field of view of a satellite.
Straightforwardly, the marginal return of including a satellite in the selection is the additional area that would be covered by that satellite, that is not already covered by the other satellites in the selection.

We test varying combinations of $r_{i}$ and $F$ over a horizon of $100$ time steps, at each time step propagating the satellites in their orbit and subsequently running DRG to reconstruct a solution.
Figure~\ref{fig:various_ground_coverage} displays visualizations for the resulting ground coverage obtained using a sampling set size of $r_{i}=120$, along with varying $F$ values setting the coverage threshold.
Tables~\ref{tab:average_budget_drg_eo} and~\ref{tab:average_sol_time_drg_eo} respectively display the average budget cost incurred and the average wall-clock computation time that DRG takes for each combination of $r_{i}$ and $F$.
We observe from the results that increasing the size of the sampling set in Algorithm~\ref{alg:dg} generally reduces the average cost of the satellite selection, but only marginally.
However, it is clear from Table~\ref{tab:average_sol_time_drg_eo} that increasing the value of $r_{i}$ causes a significant increase in the wall-clock computation time needed for DRG to produce a solution.
The results of both Sections \ref{subsec:kf_example} and \ref{subsec:coverage quality} demonstrate the benefit of incorporating randomization into the standard \textsc{Greedy} algorithm, as nearly identical costs for the selection budget can be achieved with a much lower computation time.

\subsection{Multi-Objective Robustness with \textsc{Random-WSSA}}
\begin{figure}
     \centering
     \begin{subfigure}[b]{\columnwidth}
         \centering
         \includegraphics[width=.9\textwidth]{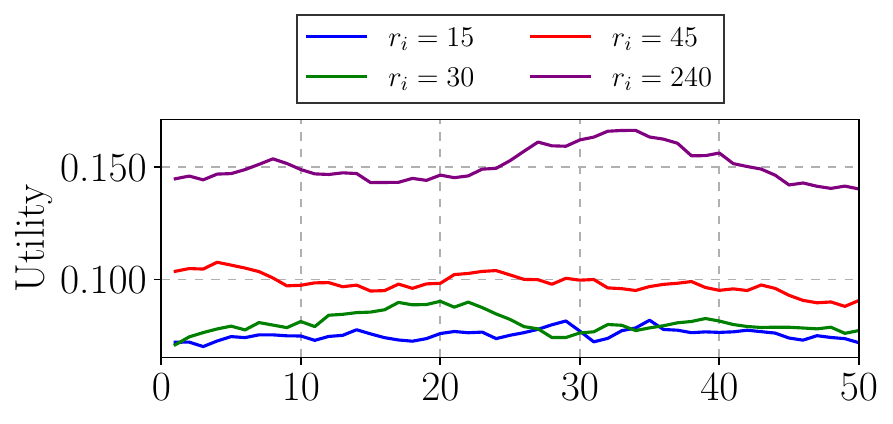}
     \end{subfigure}
     \begin{subfigure}[b]{\columnwidth}
         \centering
         \includegraphics[width=.9\textwidth]{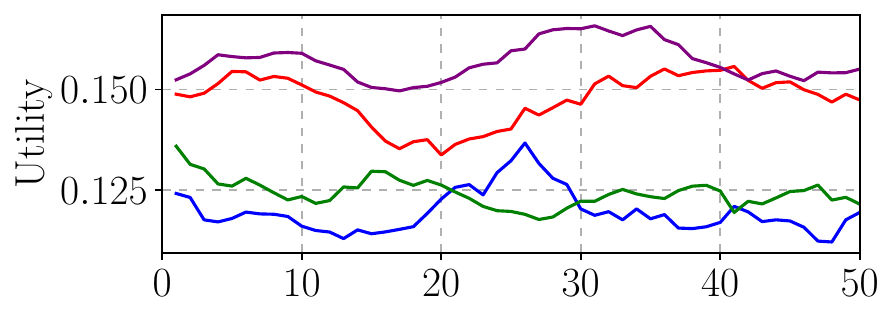}
     \end{subfigure}
     \begin{subfigure}[b]{\columnwidth}
         \centering
         \includegraphics[width=.9\textwidth]{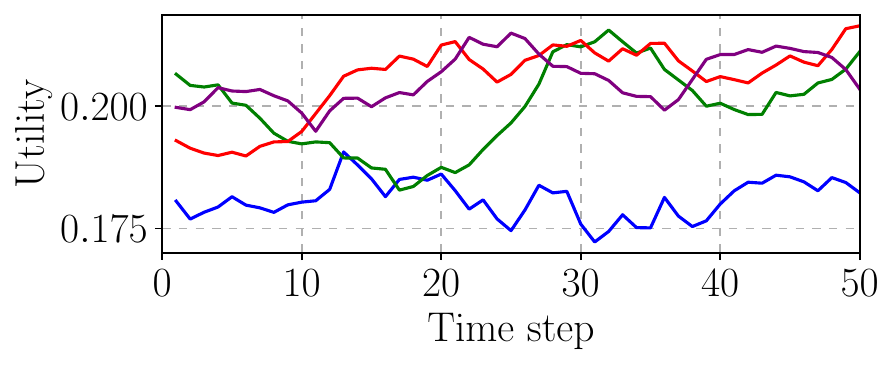}
     \end{subfigure}
    \caption{Total utility over time for various $r_{i}$ and $B$ for the \textsc{Random-WSSA}. From top to bottom are the plots for $B=10$, $15$, and $20$, and $100$, respectively. The results are put through a moving average filter with a window size of $10$. Note that $r_i =240$ corresponds to using the entire ground set, i.e., the standard SSA.}
    \label{fig:wssa_results}
\end{figure}
\begin{table}[t]
    \centering
    \caption{Average solution time in seconds for \textsc{Random-WSSA} for various $r_{i}$ and $B$. Note that $r_i =240$ corresponds to using the entire ground set, i.e., the standard SSA.}
    \begin{NiceTabular}{|c||c|c|c|}
        \hline
        ${ }$ & $B = 10$ & $B = 15$ & $B = 20$ \\ \hline \hline
        $r_{i} = 15$ & $39.91$ & $77.56$ & $132.80$ \\ \hline
        $r_{i} = 30$ & $84.17$ & $159.74$ & $275.86 $ \\ \hline
        $r_{i} = 45$ & $128.37$ & $265.39$ & $427.67$ \\ \hline
        $r_{i} = 240$ & $647.69$ & $1559.20$ & $2461.96$ \\ \hline
    \end{NiceTabular}
    \label{tab:wssa_times}
\end{table}
Lastly, we move on to experimentally corroborating the results we obtain in Section \ref{sec:robust}. To reiterate, we deal with a notion of robustness that is ubiquitous in submodular optimization~\cite{RSOS,ben2009robust,ben2002robust,NEURIPS2018_7e448ed9}, wherein we have multiple objectives and aim to produce a solution that performs well with respect to each, subject to a budget constraint. 

We instantiate a constellation with the same parameters as in the previous subsection, tasked with accomplishing six distinct tasks. The first five tasks $f^1, \ldots, f^5$ are atmospheric sensing tasks as in the experiments of Section \ref{subsec:kf_example}. As opposed to that scenario, each task instantiates five randomly located points on Earth. The utility values $f^1(\S), \ldots, f^5(\S)$ of these tasks are proportionate to the additive inverse of the mean-squared error (MSE) achieved by the selection $\S$ of satellites.

The sixth task, $f^6$, involves ground coverage of the Earth. The utility value $f^6(\S)$ of this task is proportionate to the Earth coverage achieved by the selected subset $\S$ of the satellites in the constellation. The area of coverage is determined based on the cell-grid structure explained in Subsection \ref{subsec:coverage quality}.

It is important to note that the utility of all six tasks $f^1, \ldots, f^6$ is normalized to the range of $[0, 1]$, by dividing the utility $f^i(\S)$ of a selection $\S$ on task $i$ at any time step by $f^i(\N)$, the maximum utility achievable by selecting the entire ground set at that time step. This ensures that the importance scores of individual tasks are not skewed by the arbitrary value of their utilities.

We run the simulation for $50$ time steps and make robust selections using SSA and \textsc{Random-WSSA} with various random sampling set sizes. Figure \ref{fig:wssa_results} shows the comparison in terms of the utility achieved by the two algorithms, where the case of $r_i=240$ corresponds to the standard SSA, and $r_i=15$, $r_i=30$ and $r_i=45$ correspond to running \textsc{Random-WSSA} with various sampling set sizes. The results indicate that \textsc{Random-WSSA} achieves comparable performance to SSA as the size of the sampling set increases, especially in lower-budget settings. However, as is the case with the MRG algorithm, as the budget $B$ increases, we observe that the dependence on $r_i$ weakens, as the increase in budget allows for more room to retroactively correct suboptimal selections. It is also worth noting that both algorithms are run with $\alpha=1$, i.e., the non-relaxed version of the problem in terms of the budget constraint. Nevertheless, they are consistently able to produce solutions at each time step of the simulation.

Table \ref{tab:wssa_times} demonstrates the solution times for \textsc{Random-WSSA} for various combinations of $r_i$ and $B$. It is clear, in this case, that \textsc{Random-WSSA} enjoys significantly better computational complexity in comparison to SSA, taking nearly twenty times as less time as SSA in some cases, whereas SSA takes longer than $10$ minutes on average for a single iteration even in the low-budget scenario of $B=10$. It is straightforward to conclude, then, that SSA is undesirable in time-critical conditions and that our proposed \textsc{Random-WSSA} is much better-suited to such scenarios.

\section{Conclusion}
We introduced three novel algorithms, namely the \textsc{Modified Randomized Greedy} (MRG), the \textsc{Dual Randomized Greedy} (DRG), and \textsc{Randomized Weak Submodular Saturation Algorithm (Random-WSSA)}. The first two of these algorithms aim to address weak-submodular maximization problems under budget and performance constraints, respectively. The last aims to achieve multi-objective robustness in the presence of performance constraints. Subsequently, we provided theoretical guarantees that hold with high probability on the performance of these algorithms. We illustrated the efficacy of the three algorithms through a practical scenario involving a sensors selection problem, where the sensors in question are satellites in a LEO constellation.

A clear avenue for potential research is the further investigation of the martingale assumption. One could explore whether domain-specific knowledge can enhance the tightness of the obtained bounds, or whether a guarantee that holds on expectation for DRG could be established, with the aim of using it in conjunction with a post-optimization-based method using multiple parallel instances of the algorithm to obtain a guarantee that holds with high probability.
\newpage
\bibliographystyle{ieeetr}
\bibliography{refs}

@inproceedings{hashemi2017randomized,
	title={A randomized greedy algorithm for near-optimal sensor scheduling in large-scale sensor networks},
	author={Hashemi, Abolfazl and Ghasemi, Mahsa and Vikalo, Haris and Topcu, Ufuk},
	booktitle={American Control Conference (ACC)},
	pages={1027--1032},
	year={2018},
	organization={IEEE}
}

@inproceedings{hashemi2019submodular,
  title={Submodular Observation Selection and Information Gathering for Quadratic Models},
  author={Hashemi, Abolfazl and Ghasemi, Mahsa and Vikalo, Haris},
  booktitle={Proceedings of the 36th International Conference on Machine Learning},
  volume={97},
  year={2019}
}

@article{hashemi2020randomized,
  title={Randomized greedy sensor selection: Leveraging weak submodularity},
  author={Hashemi, Abolfazl and Ghasemi, Mahsa and Vikalo, Haris and Topcu, Ufuk},
  journal={IEEE Transactions on Automatic Control},
  volume={66},
  number={1},
  pages={199--212},
  year={2021},
  publisher={IEEE}
}

@inproceedings{ghasemi2019submodularity,
  title={On Submodularity of Quadratic Observation Selection in Constrained Networked Sensing Systems},
  author={Ghasemi, Mahsa and Hashemi, Abolfazl and Topcu, Ufuk and Vikalo, Haris},
  booktitle={2019 American Control Conference (ACC)},
  pages={4671--4676},
  year={2019},
  organization={IEEE}
}

@inproceedings{hashemi2021performance,
  title={On the Performance-Complexity Tradeoff in Stochastic Greedy Weak Submodular Optimization},
  author={Hashemi, Abolfazl and Vikalo, Haris and de Veciana, Gustavo},
  booktitle={ICASSP 2021-2021 IEEE International Conference on Acoustics, Speech and Signal Processing (ICASSP)},
  pages={3540--3544},
  year={2021},
  organization={IEEE}
}

@article{hashemi2021PSG,
  title={On the Benefits of Progressively Increasing Sampling Sizes in Stochastic Greedy Weak Submodular Maximization},
  author={Hashemi, Abolfazl and Vikalo, Haris and de Veciana, Gustavo},
  journal={IEEE Transactions on Signal Processing},
  volume={70},
  pages={3978--3992},
  year={2022},
  publisher={IEEE}
}

@article{wolsey1982analysis,
  title={An analysis of the greedy algorithm for the submodular set covering problem},
  author={Wolsey, Laurence A},
  journal={Combinatorica},
  volume={2},
  number={4},
  pages={385--393},
  year={1982},
  publisher={Springer}
}

@article{khuller1999budgeted,
  title={The budgeted maximum coverage problem},
  author={Khuller, Samir and Moss, Anna and Naor, Joseph Seffi},
  journal={Information processing letters},
  volume={70},
  number={1},
  pages={39--45},
  year={1999},
  publisher={Elsevier}
}

@inproceedings{hassidim2017robust,
  title={Robust guarantees of stochastic greedy algorithms},
  author={Hassidim, Avinatan and Singer, Yaron},
  booktitle={International Conference on Machine Learning},
  pages={1424--1432},
  year={2017},
  organization={PMLR}
}

@article{nemhauser1978analysis,
  title={An analysis of approximations for maximizing submodular set functions},
  author={Nemhauser, George L and Wolsey, Laurence A and Fisher, Marshall L},
  journal={Mathematical Programming},
  volume={14},
  number={1},
  pages={265--294},
  year={1978},
  publisher={Springer}
}

@inproceedings{das2011submodular,
  title={Submodular meets Spectral: {G}reedy Algorithms for Subset Selection, Sparse Approximation and Dictionary Selection},
  author={Das, Abhimanyu and Kempe, David},
  booktitle={Proceedings of the International Conference on Machine Learning (ICML)},
  pages={1057--1064},
  year={2011}
}

@inproceedings{mirzasoleiman2015lazier,
  title={Lazier Than Lazy Greedy},
  author={Mirzasoleiman, Baharan and Badanidiyuru, Ashwinkumar and Karbasi, Amin and Vondrak, Jan and Krause, Andreas},
  booktitle={AAAI Conference on Artificial Intelligence},
  year={2015},
  organization={AAAI}
}

@article{damavandi2015robust,
  title={Robust meter placement for state estimation in active distribution systems},
  author={Damavandi, Mohammad Ghasemi and Krishnamurthy, Vikram and Mart{\'\i}, Jos{\'e} R},
  journal={IEEE Transactions on Smart Grid},
  volume={6},
  number={4},
  pages={1972--1982},
  year={2015},
  publisher={IEEE}
}

@inproceedings{horel2016maximization,
  title={Maximization of approximately submodular functions},
  author={Horel, Thibaut and Singer, Yaron},
  booktitle={Advances in Neural Information Processing Systems (NIPS)},
  pages={3045--3053},
  year={2016}
}

@book{williamson2011design,
	title={The design of approximation algorithms},
	author={Williamson, David P and Shmoys, David B},
	year={2011},
	publisher={Cambridge university press}
}

@inproceedings{zhang2016submodular,
	title={Submodular optimization with routing constraints},
	author={Zhang, Haifeng and Vorobeychik, Yevgeniy},
	booktitle={AAAI Conference on Artificial Intelligence},
	year={2016}
}

@article{chung2006concentration,
  title={Concentration inequalities and martingale inequalities: a survey},
  author={Chung, Fan and Lu, Linyuan},
  journal={Internet Mathematics},
  volume={3},
  number={1},
  pages={79--127},
  year={2006},
  publisher={Taylor \& Francis}
}

@inproceedings{chamon2017approximate,
  title={Approximate supermodularity bounds for experimental design},
  author={Chamon, Luiz and Ribeiro, Alejandro},
  booktitle={Advances in Neural Information Processing Systems (NIPS)},
  pages={5409--5418},
  year={2017}
}

@inproceedings{bian2017guarantees,
	title={Guarantees for greedy maximization of non-submodular functions with applications},
	author={Bian, Andrew An and Buhmann, Joachim M and Krause, Andreas and Tschiatschek, Sebastian},
	booktitle={International Conference on Machine Learning (ICML)},
	pages={498--507},
	year={2017},
	organization={Omnipress}
}

@article{feige1998threshold,
	title={A threshold of ln n for approximating set cover},
	author={Feige, Uriel},
	journal={Journal of the ACM},
	volume={45},
	number={4},
	pages={634--652},
	year={Jul. 1998},
	publisher={ACM}
}

@article{elenberg2016restricted,
  title={Restricted strong convexity implies weak submodularity},
author={Elenberg, Ethan R and Khanna, Rajiv and Dimakis, Alexandros G and Negahban, Sahand},
journal={The Annals of Statistics},
volume={46},
number={6B},
pages={3539--3568},
year={2018},
publisher={Institute of Mathematical Statistics}
}

@incollection{krause2014submodular,
	title={Submodular Function Maximization},
	author={Krause, Andreas and Golovin, Daniel},
	booktitle={Tractability: Practical Approaches to Hard Problems},
	pages={71--104},
	year={2014},
	publisher={Cambridge University Press}
}

@inproceedings{khanna2017scalable,
  title={Scalable Greedy Feature Selection via Weak Submodularity},
  author={Khanna, Rajiv and Elenberg, Ethan and Dimakis, Alex and Negahban, Sahand and Ghosh, Joydeep},
  booktitle={Artificial Intelligence and Statistics},
  pages={1560--1568},
  year={2017}
}

@inproceedings{liu2016towards,
  title={Towards scalable voltage control in smart grid: a submodular optimization approach},
  author={Liu, Zhipeng and Clark, Andrew and Lee, Phillip and Bushnell, Linda and Kirschen, Daniel and Poovendran, Radha},
  booktitle={Proceedings of the 7th International Conference on Cyber-Physical Systems},
  pages={20},
  year={2016},
  organization={IEEE Press}
}

@article{walker1984satellite,
  title={Satellite constellations},
  author={Walker, John G},
  journal={Journal of the British Interplanetary Society},
  volume={37},
  pages={559},
  year={1984}
}

@book{bar2004estimation,
  title={Estimation with applications to tracking and navigation: theory algorithms and software},
  author={Bar-Shalom, Yaakov and Li, X Rong and Kirubarajan, Thiagalingam},
  year={2004},
  publisher={John Wiley \& Sons}
}

@article{lorenz1963deterministic,
  title={Deterministic nonperiodic flow},
  author={Lorenz, Edward N},
  journal={Journal of atmospheric sciences},
  volume={20},
  number={2},
  pages={130--141},
  year={1963}
}

@article{poghosyan2017cubesat,
  title={Cube{S}at evolution: Analyzing {C}ube{S}at capabilities for conducting science missions},
  author={Poghosyan, Armen and Golkar, Alessandro},
  journal={Progress in Aerospace Sciences},
  volume={88},
  pages={59--83},
  year={2017},
  publisher={Elsevier}
}

@article{wagner2021distributed,
  title={Distributed space missions applied to sea surface height monitoring},
  author={Wagner, Katherine M and Schroeder, Kevin K and Black, Jonathan T},
  journal={Acta Astronautica},
  volume={178},
  pages={634--644},
  year={2021},
  publisher={Elsevier}
}

@article{cahoy2017initial,
  title={Initial results from ACCESS: an autonomous cubesat constellation scheduling system for earth observation},
  author={Cahoy, Kerri and Kennedy, Andrew K},
  year={2017}
}

@inproceedings{mandl2016hyperspectral,
  title={Hyperspectral CubeSat constellation for natural hazard response},
  author={Mandl, Daniel and Crum, Gary and Ly, Vuong and Handy, Matthew and Huemmrich, Karl F and Ong, Lawrence and Holt, Ben and Maharaja, Rishabh},
  booktitle={Annual AIAA/USu conference on small satellites},
  number={GSFC-E-DAA-TN33317},
  year={2016}
}

@article{santilli2018cubesat,
  title={Cube{S}at constellations for disaster management in remote areas},
  author={Santilli, Giancarlo and Vendittozzi, Cristian and Cappelletti, Chantal and Battistini, Simone and Gessini, Paolo},
  journal={Acta Astronautica},
  volume={145},
  pages={11--17},
  year={2018},
  publisher={Elsevier}
}

@article{nag2016cubesat,
  title={Cube{S}at constellation design for air traffic monitoring},
  author={Nag, Sreeja and Rios, Joseph L and Gerhardt, David and Pham, Camvu},
  journal={Acta Astronautica},
  volume={128},
  pages={180--193},
  year={2016},
  publisher={Elsevier}
}

@article{wu2021multiple,
  title={A multiple-{C}ube{S}at constellation for integrated earth observation and marine/air traffic monitoring},
  author={Wu, Shufan and Chen, Wen and Cao, Caixia and Zhang, Chuanxin and Mu, Zhongcheng},
  journal={Advances in Space Research},
  volume={67},
  number={11},
  pages={3712--3724},
  year={2021},
  publisher={Elsevier}
}

@article{selva2012survey,
  title={A survey and assessment of the capabilities of {C}ube{S}ats for {E}arth observation},
  author={Selva, Daniel and Krejci, David},
  journal={Acta Astronautica},
  volume={74},
  pages={50--68},
  year={2012},
  publisher={Elsevier}
}

@article{blackwell2018overview,
  title={An overview of the {TROPICS NASA} {E}arth venture mission},
  author={Blackwell, William J and Braun, S and Bennartz, R and Velden, C and DeMaria, M and Atlas, R and Dunion, J and Marks, F and Rogers, R and Annane, B and others},
  journal={Quarterly Journal of the Royal Meteorological Society},
  volume={144},
  pages={16--26},
  year={2018},
  publisher={Wiley Online Library}
}

@article{boshuizen2014results,
  title={Results from the planet labs flock constellation},
  author={Boshuizen, Christopher and Mason, James and Klupar, Pete and Spanhake, Shannon},
  year={2014}
}

@article{siewert1995system,
  title={A system architecture to advance small satellite mission operations autonomy},
  author={Siewert, Sam and McClure, Linden},
  year={1995}
}

@article{attigeri2019feature,
  title={Feature selection using submodular approach for financial big data},
  author={Attigeri, Girija and Manohara Pai, MM and Pai, Radhika M},
  journal={Journal of Information Processing Systems},
  volume={15},
  number={6},
  pages={1306--1325},
  year={2019},
  publisher={Korea Information Processing Society}
}

@inproceedings{richards2001distributed,
  title={Distributed Satellite Constellation Planning and Scheduling.},
  author={Richards, Robert A and Houlette, Ryan T and Mohammed, John L and others},
  booktitle={FLAIRS Conference},
  pages={68--72},
  year={2001}
}

@article{hibbard_hashemi_tanaka_topcu_2023,
  title={Randomized Greedy Algorithms for Sensor Selection in Large-Scale Satellite Constellations},
  author={Michael Hibbard and Abolfazl Hashemi and Takashi Tanaka and Ufuk Topcu},
  journal={Proceedings of the American Control Conference},
  year={2023},
}

@article{RSOS,
  author  = {Andreas Krause and H. Brendan McMahan and Carlos Guestrin and Anupam Gupta},
  title   = {Robust Submodular Observation Selection},
  journal = {Journal of Machine Learning Research},
  year    = {2008},
  volume  = {9},
  number  = {93},
  pages   = {2761--2801},
  url     = {http://jmlr.org/papers/v9/krause08b.html}
}

@article{wolsey,
	author = {Wolsey, L.  A. },
	journal = {Combinatorica},
	number = {4},
	pages = {385--393},
	title = {An analysis of the greedy algorithm for the submodular set covering problem},
	volume = {2},
	year = {1982}}

@inproceedings{cardinality, author = {Buchbinder, Niv and Feldman, Moran and Naor, Joseph (Seffi) and Schwartz, Roy}, title = {Submodular Maximization with Cardinality Constraints}, year = {2014}, isbn = {9781611973389}, publisher = {Society for Industrial and Applied Mathematics}, address = {USA}, booktitle = {Proceedings of the Twenty-Fifth Annual ACM-SIAM Symposium on Discrete Algorithms}, pages = {1433–1452}, numpages = {20}, location = {Portland, Oregon}, series = {SODA '14} }

@inproceedings{revenue,
	address = {Berlin, Heidelberg},
	author = {Dughmi, Shaddin and Roughgarden, Tim and Sundararajan, Mukund},
	booktitle = {Auctions, Market Mechanisms and Their Applications},
	editor = {Das, Sanmay and Ostrovsky, Michael and Pennock, David and Szymanksi, Boleslaw},
	pages = {89--91},
	publisher = {Springer Berlin Heidelberg},
	title = {Revenue Submodularity},
	year = {2009}}

@article{gametheory,
	author = {Andreas S. Schulz and Nelson A. Uhan},
	journal = {Discrete Optimization},
	number = {2},
	pages = {163-180},
	title = {Approximating the least core value and least core of cooperative games with supermodular costs},
	volume = {10},
	year = {2013}}

@book{ben2009robust,
  title={Robust optimization},
  author={Ben-Tal, Aharon and El Ghaoui, Laurent and Nemirovski, Arkadi},
  volume={28},
  year={2009},
  publisher={Princeton university press}
}

@article{ben2002robust,
  title={Robust optimization--methodology and applications},
  author={Ben-Tal, Aharon and Nemirovski, Arkadi},
  journal={Mathematical programming},
  volume={92},
  pages={453--480},
  year={2002},
  publisher={Springer}
}

@inproceedings{NEURIPS2018_7e448ed9,
	author = {Udwani, Rajan},
	booktitle = {Advances in Neural Information Processing Systems},
	editor = {S. Bengio and H. Wallach and H. Larochelle and K. Grauman and N. Cesa-Bianchi and R. Garnett},
	publisher = {Curran Associates, Inc.},
	title = {Multi-objective Maximization of Monotone Submodular Functions with Cardinality Constraint},
	volume = {31},
	year = {2018}}

@misc{harris2022martingale,
      title={The martingale Z-test}, 
      author={Kenneth D. Harris},
      year={2022},
      eprint={2207.02893},
      archivePrefix={arXiv},
      primaryClass={stat.ME}
}

@article{JMLR:v19:16-534,
  author  = {Abhimanyu Das and David Kempe},
  title   = {Approximate Submodularity and its Applications: Subset Selection, Sparse Approximation and Dictionary Selection},
  journal = {Journal of Machine Learning Research},
  year    = {2018},
  volume  = {19},
  number  = {3},
  pages   = {1--34},
  url     = {http://jmlr.org/papers/v19/16-534.html}
}

@article{kaya2025randomized,
  title={Randomized greedy methods for weak submodular sensor selection with robustness considerations},
  author={Kaya, Ege Can and Hibbard, Michael and Tanaka, Takashi and Topcu, Ufuk and Hashemi, Abolfazl},
  journal={Automatica},
  volume={171},
  pages={111984},
  year={2025},
  publisher={Elsevier}
}
\end{document}